\title{Suppression of Chemotactic Blowup by Strong Buoyancy in Stokes-Boussinesq Flow with Cold Boundary}
\author{Zhongtian Hu\thanks{
                  Department of Mathematics, Duke University, email: zhongtian.hu@duke.edu}
        \and
        Alexander Kiselev\thanks{
                  Department of Mathematics, Duke University, email: kiselev@math.duke.edu
}
        }
\documentclass[11pt]{article}
\usepackage[margin=1in]{geometry}

\usepackage{comment,slashed,tensor}
\usepackage[T1]{fontenc}
\usepackage{ stmaryrd }

 \usepackage[bookmarksnumbered, bookmarksopen, colorlinks, citecolor=red, linkcolor=black]{hyperref}
 \usepackage{amsfonts}
\usepackage{graphicx}
 \usepackage{amsmath, amssymb}
 \usepackage{amsthm}
 \usepackage{mathtools}
\usepackage[noabbrev]{cleveref}

\numberwithin{equation}{section}

 \usepackage{comment}
 \usepackage{amsthm}
\newtheorem{thm}{Theorem}[section]
\newtheorem{lem}{Lemma}[section]
\newtheorem{cor}{Corollary}[section]
\newtheorem{prop}{Proposition}[section]
\newtheorem*{prop*}{Proposition}
\newtheorem{rmk}{Remark}[section]
\newtheorem{defn}{Definition}[section]

\newtheorem{aprop}{Proposition}[section]
\newtheorem{athm}{Theorem}[section]

\newcommand{\bu}{\bar{u}}
\newcommand{\Q}{\mathbb{Q}}
\newcommand{\bP}{\mathbb{P}}
\newcommand{\calA}{\mathcal{A}}
\newcommand{\R}{\mathbb{R}}
\newcommand{\N}{\mathbb{N}}

\newcommand{\T}{\mathbb{T}}

\newcommand{\rhon}{\rho^{(n)}}
\newcommand{\un}{u^{(n)}}

\DeclareMathOperator{\divv}{div}

\DeclareMathOperator{\supp}{supp}
\DeclareMathOperator{\diam}{diam}

\newcommand{\srho}[1]{\p_t^{#1}\rho}
\newcommand{\su}[1]{\p_t^{#1}u}

\newcommand{\p}{\partial}

\begin{document}
\newpage
\maketitle
%%%%%%%%%%%%%%%%%%%%%%%%%%%
% abstract, keywords and Subject classification are optional.
%%%%%%%%%%%%%%%%%%%%%%%%%%%
\begin{abstract}
In this paper, we show that the Keller-Segel equation equipped with zero Dirichlet Boundary condition and actively coupled to a Stokes-Boussinesq flow is globally well-posed provided that the coupling is sufficiently large. We will in fact show that the dynamics is quenched after certain time. In particular,
such active coupling is blowup-suppressing in the sense that it enforces global regularity for some initial data leading to a finite-time singularity when the flow is absent.
\end{abstract}

\section{Introduction}

The Keller-Segel equation is a well known model of chemotaxis \cite{KS,Patlak}.
It describes a population of bacteria or slime mold
that move in response to attractive external chemical that they themselves secrete.
The equation has interesting analytical properties: its solutions can form mass concentration singularities
in dimension greater than one (see e.g. \cite{Pert}) where further references can be found).
Often, chemotactic processes take place in ambient fluid. One natural question is then how the presence of fluid flow can affect singularity formation.
In the case where the ambient flow is passive - prescribed and independent of the bacteria density - it has been shown that presence of the flow
can suppress singularity formation. The flows that have been analyzed include some flows with strong mixing properties \cite{KX},
shear flows \cite{BH}, hyperbolic splitting flow \cite{HT}, and some cellular flows \cite{GXZ}. In a similar vein, \cite{CDFM} explored advection induced regularity has been
for the Kuramoto-Sivashinsky equation.

The case where the fluid flow is active - satisfies some fluid equation driven by force exerted the bacteria - is very interesting but harder to analyze.
There have been many impressive works that analyzed such coupled systems, usually via buoyancy force;
see for example \cite{di2010chemotaxis,duan2010global,lorz2010coupled,
liu2011coupled,lorz2012coupled,Winkler2012, CKL, DX, TW, Winkler2021} where further references can be found.
in some cases results involving global existence of regular solutions
 (the precise notion of their regularity is different in different papers) have been proved.
These results, however, apply either in the settings
where the initial data satisfy some smallness assumptions (e.g. \cite{duan2010global, lorz2012coupled, CKL})
or in the systems where both fluid and chemotaxis equations may not form a singularity if not coupled (e.g. \cite{Winkler2012,  TW, Winkler2021}).
Recently, in \cite{Hact} and \cite{zzz}, the authors analyzed Patlak-Keller-Segel equation coupled to the Navier-Stokes equation near Couette flow.
Based on ideas of blowup suppression in shear flows and stability of the Couette flow,
the authors proved that global regularity can be enforced if the amplitude of the Couette flow is dominantly large and if the initial flow is very close to it.
The density/fluid coupling in these works is not by buoyancy force but instead involves a model of the swimmer's effect on fluid that leads to special algebraic properties of the system.

In the recent work of the authors joint with Yao \cite{HKY}, the two dimensional Keller-Segel equation coupled with the incompressible porous media via buoyancy force has been analyzed.
It has been proved that in this case, an arbitrary weak coupling constant (i.e, gravity) completely regularizes the system, and the solutions
become globally regular for any reasonable initial data. At the heart of the proof is the analysis of potential energy, whose time derivative includes a coercive "main term" $\|\partial_{x_1}\rho\|_{H^{-1}_0}^2$
(where $\rho$ is the bacteria density). Essentially, this $H^{-1}_0$ norm has to become small, and intuitively this implies mixing in the $x_1$ direction.
Hence the solution becomes in some sense quasi-one-dimensional and this arrests singularity formation.

Our goal in this paper is to analyze the Keller-Segel equation in an arbitrary smooth domain in dimensions two and three coupled to the Stokes flow via buoyancy force:
\begin{equation}
    \label{eq:ksstokes}
    \begin{cases}
    \p_t \rho + u\cdot \nabla \rho -\Delta \rho + \divv(\rho\nabla(-\Delta)^{-1}\rho) = 0,& x \in \Omega,\\
    \p_tu-\Delta u + \nabla p = g\rho  e_z,\; \divv u = 0,& x\in \Omega,\\
    u(0,x) = u_0(x),\; \rho(0,x) = \rho_0(x),\; \rho_0(x) \ge 0,\\
    u|_{\p \Omega} = 0,\; \rho|_{\p \Omega} = 0.
    \end{cases}
\end{equation}
Here $\Omega$ is a smooth, compact domain in $\R^d$, $d= 2$ or $3$. $e_z$ denotes the unit vector $(0,1)$ when $d=2$ or $(0,0,1)$ when $d=3$. $g \in \R^+$ is the Rayleigh number representing the buoyancy strength. Moreover, the operator $(-\Delta)^{-1}$ denotes the inverse homogeneous Dirichlet Laplacian corresponding to the domain $\Omega$. In the case of the Stokes flow, the fluid velocity is more regular, and the equation includes time derivative that complicates matters, partly due to a loss of a ``Biot-Savart law'' that relates $\rho$ and $u$ directly. We are unable to prove global regularity for all $g,$ and we are not sure if it is true.
Our main result is global regularity for strong buoyancy. The proof is completely different from \cite{HKY}: it relies on softer arguments and the analysis of the large buoyancy limit.

The first part of this paper addresses the local well-posedness of regular solutions to \eqref{eq:ksstokes}. Before we make precise of the notion of a \textit{regular solution}, we shall first introduce the following useful function spaces: to study the regularity properties of $\rho$, we consider
\begin{align*}
H^1_0 &:= \text{completion of }C^\infty_c(\Omega) \text{ with respect to }H^1 \text{ norm},\\
H^{-1}_0 &:= \text{dual space of }H^1_0.
\end{align*}
Moreover, we use the traditional notation $W^{k,p}(\Omega)$ to denote Sobolev spaces equipped with norm $\|\cdot\|_{k,p}$ in domain $\Omega$. If $p = 2$, we in particular write $H^s(\Omega) = W^{s,2}(\Omega)$ equipped with norm $\|\cdot\|_s$. We will write $W^{k,p}$ (or $H^s$) instead of $W^{k,p}(\Omega)$ (or $H^s(\Omega)$) for simplicity if there is no confusion over the domain involved. We also say an $n$-vector field $v = (v_i)_i \in H^s$ if $v_i \in H^s$ for $i = 1,\hdots, n$.

As we also need to work with Stokes equation, it is standard to introduce the following spaces:
$$
C^\infty_{c,\sigma} := \{u \in C_c^\infty(\Omega)\;|\; \divv u = 0\},
$$
$$
H := \text{completion of }C^\infty_{c,\sigma} \text{ with respect to }L^2 \text{ norm},
$$
$$
V := H \cap H^1_0(\Omega),\; V^*:= \text{dual space of }V,
$$
where $V$ is equipped with $H^1_0$ norm, and $V^*$ is equipped with the standard dual norm. We also recall the following useful operators: the Leray projector $\bP: L^2 \to H$ and the Stokes operator $\calA := - \bP\Delta: D(\calA) = H^2\cap V \to H$. We refer the readers to \cite{CF} for a more thorough treatment of such operators. As a common practice in the study of Stokes equation, one may equivalently rewrite the fluid equation as:
\begin{equation}
    \label{eq:stokes2}
    \p_t u + \calA u = g\bP(\rho e_z),
\end{equation}
We will often use this formulation in regularity estimates for the rest of this work.

Now we give a rigorous definition of a \textit{regular solution} to \eqref{eq:ksstokes}.
\begin{defn}
Given initial data $\rho_0 \in H^1_0$, $u_0 \in V$, and some $T > 0$, we say the pair $(\rho(t,x), u(t,x))$ is a regular solution to \eqref{eq:ksstokes} on $[0,T]$ if
\begin{align*}
\rho \in C^0([0,T]; H^1_0) \cap L^2((0,T); H^2\cap H^1_0),&\; u \in C^0([0,T]; V) \cap L^2((0,T); H^2\cap V),\\
\p_t \rho \in C^0([0,T];H^{-1}_0),&\; \p_t u \in C^0([0,T]; V^*),\\
\rho \in C^\infty((0,T] \times \Omega),&\; u \in C^\infty((0,T] \times \Omega).
\end{align*}
\end{defn}

With this definition, we are able to obtain the following well-posedness result:
\begin{thm}
\label{thm:lwp}
Given initial data $\rho_0 \in H^1_0$, $u_0 \in V$, there exists a $T_* = T_*(\rho_0) > 0$ such that there exists a unique regular solution $(\rho, u)$ to problem \eqref{eq:ksstokes} on $[0,T_*]$.
\end{thm}

We will then prove a regularity criterion which allows us to continue the regular solution of \eqref{eq:ksstokes} as long as the $L^{\frac{4}{4-d}}_tL^2_x$ norm of $\rho$ is controlled. More precisely, we have
\begin{thm}
\label{thm:L2criterion}
Let $\Omega \subset \R^d$, $d = 2,3$, be a smooth, bounded domain. If the maximal lifespan $T_0$ of the regular solution $(\rho, u)$ to problem \eqref{eq:ksstokes} is finite, then necessarily
$$
\lim_{t \nearrow T_0}\int_0^t\|\rho\|_{L^2}^{\frac{4}{4-d}}ds = \infty.
$$
\end{thm}
A similar result was proved in \cite{KX} in the periodic setting for the uncoupled Keller-Segel equation.

In the second part of this work, we will quantify the quenching effect of the Stokes-Boussinesq flow with strong buoyancy on the Keller-Segel equation equipped with homogeneous Dirichlet boundary condition. To be more precise, we show that the flow can suppress the norm $\|\rho\|_{L^2}$ to be sufficiently small within the time scale of local existence. In particular, we will show the following main result of this work:
\begin{thm}
    \label{thm:main}
    For any smooth, bounded domain $\Omega \subset \R^d$, $d = 2,3$, and arbitrary initial data $\rho_0 \in H^1_0$, $u_0 \in V$, there exists $g_* = g_*(\rho_0, u_0)$ so that for any $g \ge g_*$, \eqref{eq:ksstokes} admits a regular, global-in-time solution. In particular, $\rho$ is quenched exponentially fast in the sense that
    \begin{equation}
    \label{eq:quenching}
    \lim_{t \to \infty} e^{ct}\|\rho(t)\|_{L^2} \le C,
    \end{equation}
    where $c$, $C$ are positive constants that only depend on the domain $\Omega$.
\end{thm}

We observe that if we fix any smooth passive divergence-free $u$ satisfying the no-flux $u \cdot n = 0$ boundary condition, then one can find smooth initial data $\rho_0$
such that the solution of the first equation in \eqref{eq:ksstokes} will lead to finite time blow up. The argument proving this is very similar to that of Theorem 8.1 in \cite{KX} for the case of $\T^2$; however the localization used in the proof makes it insensitive to the boundary condition.

We will use the expression $f \lesssim g$ to denote the following: there exists some constant $C$ only depending on domain $\Omega$ such that
$
f \le Cg.
$
In particular, we will denote a generic constant depending only on $\Omega$ by $C$, and it could change from line to line. Finally, we will use the Einstein summation convention. That is, by default we sum over the repeated indices; e.g. we write $a_ix_i := \sum_i a_ix_i$.\\

{\bf Acknowledgment. } The authors acknowledge partial support of the 
NSF-DMS grants 2006372 and 2306726.

\section{Local Well-Posedness of Regular Solution}
\label{sec:lwp}
In this section, we will establish the local well-posedness of problem \eqref{eq:ksstokes}, namely Theorem \ref{thm:lwp}. It is well-known that the classical parabolic-elliptic Keller-Segel equation is locally well-posed in domains such as $\R^d$ or $ \mathbb{T}^d$, $d = 2,3$, or in a smooth, bounded domain with Neumann boundary condition on $\rho$ in suitable function spaces (see e.g. \cite{BDP,KX,Winkler2012}). However, we were unable to locate a convenient reference for
a well-posedness theorem in the scale of Sobolev spaces $H^s$ %using energy method seems to be elusive in current literature
in the scenario of \eqref{eq:ksstokes}. Thus for the sake of the completeness, we will give explicit \textit{a priori} estimates which lead to local well-posedness.

We first set up an appropriate Galerkin scheme that uses two sets of bases in Subsection \ref{subsect:galerkin}. In Subsection \ref{subsect:low}, we start with a set of lower order \textit{a priori} energy estimates which guarantee spatial regularity of a solution up to $H^2$. In Subsection \ref{subsect:high}, we will prove the existence of regular solutions by devising an inductive argument that boosts both temporal and spatial regularity up to $H^s$ for arbitrary $s$ using parabolic smoothing. In Subsection \ref{subsect:uniqueness}, we will complete the proof of Theorem \ref{thm:lwp} by showing the uniqueness of regular solutions. Finally, we will demonstrate an $L^2$ regularity criterion (i.e. Theorem \ref{thm:L2criterion}) in Subsection \ref{sec:crit}. It will be instrumental in establishing the global well-posedness of \eqref{eq:ksstokes}.

\begin{rmk}
We will only discuss the case when $d=3$. Then $d=2$ case follows from similar (and easier) arguments.
% from the $d=3$ case since $\Omega$ is a bounded domain.
\end{rmk}

\subsection{Galerkin Approximation}\label{subsect:galerkin}
Since \eqref{eq:ksstokes} is a system of semilinear parabolic equations in a compact domain, it is convenient to construct a solution to \eqref{eq:ksstokes} by Galerkin approximation. Let $\{v_k\}_k$, $\{\lambda_k\}_k$ be the eigenfunctions and eigenvalues of $-\Delta$. Let $\{w_j\}_j$, $\{\eta_j\}_j$ be the eigenfunctions and eigenvalues of the Stokes operator $\calA$. Consider the following approximate system:
\begin{equation}
\label{galerkin0}
\begin{cases}
    \p_t\rhon + \Q_n(\un\cdot \nabla\rhon) - \Delta\rhon + \Q_n(\divv(\rhon\nabla(-\Delta)^{-1}\rhon)) = 0,\\
    \p_t \un + \calA\un = g\bP_n(\rhon  e_z),\\
    \rhon(0) = \Q_n\rho_0,\; \un(0) = \bP_n u_0,
\end{cases}
\end{equation}
where $\Q_n f := (f,v_k)_{L^2}v_k$, $\bP_nf := (f, w_j)_{L^2} w_j$. Here $(\cdot,\cdot)_{L^2}$ denotes the standard $L^2$-inner product. Note that the projection operators $\bP_n, \Q_n$ are symmetric with respect to $L^2$ inner product. Writing the approximated solutions $\rhon(t,x) = \rhon_k(t)v_k(x)$, $\un(t,x) = \un_j(t)w_j(x)$ (recall that we are summing over repeated indices), we obtain the following constant-coefficient ODEs in $t$: for $l = 1,\hdots, n$,
\begin{equation}
\label{galerkin}
\begin{cases}
\frac{d}{dt}\rhon_l + C^{(n)}_{ljk}\un_j\rhon_k + \lambda_l\rhon_l - D^{(n)}_{ljk}\rhon_k\rhon_j = 0,\\
\frac{d}{dt}\un_l + \eta_l\un_l = gC_{kl}\rhon_k e_z,\\
\rhon_l(0) = (\rho_0, v_l)_{L^2},\; \un_l(0) = (u_0, w_l)_{L^2},
\end{cases}
\end{equation}
where
$$
C_{ljk}^{(n)} := (\Q_n(w_j\cdot \nabla v_k),v_l)_{L^2},\; D_{ljk}^{(n)} := \Q_n(\divv(v_k\nabla(-\Delta)^{-1}v_j), v_l)_{L^2},
$$
$$
C_{kl} := (\bP v_k, w_l)_{L^2}.
$$
\label{rmk:apriori}
To close the Galerkin approximation argument, we shall prove suitable uniform-in-$n$ energy estimates for $(\rhon, \un)$ and pass to the limit using compactness theorems. For the sake of simplicity, we shall prove such energy estimates in an \textit{a priori} fashion, for sufficiently regular solutions
of the original system \eqref{eq:ksstokes}. One could verify that all estimates below can be carried over to the approximated solutions $(\rhon, \un)$ in a straightforward manner.

\subsection{Lower Order \textit{a priori} Estimates}\label{subsect:low}
Given intial data $\rho_0 \in H^1_0, u \in V$, we first show the following $L^\infty_tL^2_x$ and $L^2_tH^1_x$ estimates for a regular solution $(\rho,u)$:

\begin{prop}
\label{prop:l2bd}
Given initial data $\rho_0 \in H^1_0, u \in V$, we assume $(\rho, u)$ is a regular solution to \eqref{eq:ksstokes} on $[0,T]$ for some $T > 0$. Then for $t \in [0,T]$, we have
\begin{equation}
\label{est:l2}
\frac{d}{dt}\|\rho\|_{L^2}^2 + \|\nabla \rho\|_{L^2}^2 \lesssim \|\rho\|_{L^2}^6,\; \frac{1}{2}\frac{d}{dt}\|u\|_{L^2}^2 + \|\nabla u\|_{L^2}^2 \le g\|u\|_{L^2}\|\rho\|_{L^2}.
\end{equation}
Moreover, there exists $T_* \in (0,1]$ only depending on $\rho_0$, and a constant $C(u_0,\rho_0,g) > 0$ such that
\begin{equation}
    \label{est:L2rho}
    \sup_{t\in [0,T_*]}\|\rho(t)\|_{L^2}^2 + \int_0^{T_*}\|\nabla\rho(t)\|_{L^2}^2 ds \le 4\|\rho_0\|_{L^2}^2.
\end{equation}
\begin{equation}
    \label{est:L2u}
    \sup_{t\in [0,T_*]}\|u(t)\|_{L^2}^2 + \int_0^{T_*}\|\nabla u(t)\|_{L^2}^2 ds \le C(\|\rho_0\|_{L^2},\|u_0\|_{L^2})(g^2 + 1).
\end{equation}
\begin{proof}
First by testing the $\rho$-equation of \eqref{eq:ksstokes} by $\rho$ and integrating by parts, we have
$$
\frac{1}{2}\frac{d}{dt}\|\rho\|_{L^2}^2 + \|\nabla \rho\|_{L^2}^2 = \frac{1}{2}\int_\Omega \rho^3 dx \le C\|\rho\|_{L^2}^{3/2}\|\nabla\rho\|_{L^2}^{3/2}\le \frac{1}{2}\|\nabla\rho\|_{L^2}^2 + C\|\rho\|_{L^2}^{6},
$$
where we used the following standard Gagliardo-Nirenberg inequality in 3D for trace-free $f$:
$$
    \|f\|_{L^3}^3 \le C\|f\|_{L^2}^{3/2}\|\nabla f\|_{L^2}^{3/2}.
$$
After rearranging, we obtain the first inequality of \eqref{est:l2}. Similarly, we test the $u$-equation of \eqref{eq:ksstokes} by $u$. After integration by parts, we have
\begin{align}
\label{aux828a}
\frac{1}{2}\frac{d}{dt}\|u\|_{L^2}^2 + \|\nabla u\|_{L^2}^2  &= g\int_\Omega u\cdot(\rho  e_z)dx\le g\|u\|_{L^2}\|\rho\|_{L^2},
\end{align}
which proves the second inequality in \eqref{est:l2}.
Then, the estimate \eqref{est:L2rho} follows immediately from applying Gr\"onwall inequality to \eqref{est:l2} and choosing $T_* = T_*(\rho_0) \le 1$ sufficiently small. Now integrating \eqref{aux828a} from $0$ to $t \in (0,T_*)$, using \eqref{est:L2rho}, and taking supremum over $t$, we have
    \begin{equation}
    \label{l2:aux2}
    \sup_{t\in [0,T_*]}\|u(t)\|_{L^2} \le 8g\|\rho_0\|_{L^2}T_* + \|u_0\|_{L^2}.
    \end{equation}
Using \eqref{est:L2rho} and \eqref{l2:aux2} in the integrated in time version of \eqref{aux828a},
%Now integrating \eqref{est:l2} and inserting \eqref{l2:aux2},
we obtain that
    \begin{equation}
    \label{l2:aux3}
    \int_0^{T_*}\|\nabla u(s)\|_{L^2}^2ds \le \|u_0\|_{L^2}^2 + 4gT_*\|\rho_0\|_{L^2}(8g\|\rho_0\|_{L^2}T_* + \|u_0\|_{L^2})
    \end{equation}
    The proof of \eqref{est:L2u} is finished after we combine \eqref{l2:aux2} and \eqref{l2:aux3}.
\end{proof}
\end{prop}
\begin{rmk}
    From now on, any appearance of $T_*$ refers to the time $T_*$ chosen in Proposition \ref{prop:l2bd}.
\end{rmk}

With Proposition \ref{prop:l2bd}, we will derive the following upgraded temporal and spatial regularity estimates for solution $(\rho,u)$ within the time interval $[0,T_*]$.

\begin{prop}
\label{prop:h1bd}
Assuming $(\rho,u)$ to be a regular solution to \eqref{eq:ksstokes} with initial data $\rho_0 \in H^1_0, u \in V$, there exists $C(\rho_0, u_0, g) > 0$ such that
\begin{align*}
    &\int_0^{T_*}\left(\|\rho(t)\|_2^2 + \|u(t)\|_2^2 + \|\p_t\rho(t)\|^2_{L^2} + \|\p_t u(t)\|^2_{L^2}\right) dt\\
    &+ \sup_{t \in [0,T_*]}(\|\rho(t)\|_1^2 + \|u(t)\|_1^2) \le C(\rho_0, u_0, g).
\end{align*}
\end{prop}

\begin{proof}
Testing the $\rho$-equation in \eqref{eq:ksstokes} by $-\Delta \rho$ and integrating by parts, we obtain:
\begin{align*}
    \frac{1}{2}\frac{d}{dt}\|\nabla \rho\|_{L^2}^2 + \|\Delta \rho\|_{L^2}^2 = \int_\Omega \Delta\rho(u\cdot\nabla \rho) + \int_\Omega \Delta\rho \divv(\rho\nabla(-\Delta)^{-1}\rho) = I + J.
\end{align*}
Let us fix $\epsilon > 0$. Using Sobolev embedding, Poincaré inequality, and Young's inequality with $\epsilon$, we can estimate $I$ by:
\begin{align*}
    I &\le \|\Delta\rho\|_{L^2}\|\nabla\rho\|_{L^2}\|u\|_{L^\infty} \le \epsilon\|\Delta\rho\|_{L^2}^2 + C(\epsilon)\|u\|_{2}^2\|\nabla \rho\|_{L^2}^2.
\end{align*}
Moreover, we can write $J$ as:
$$
J = \int_\Omega \Delta\rho \left(\nabla \rho\cdot \nabla(-\Delta)^{-1}\rho - \rho^2\right) dx = J_1 + J_2.
$$
Using the standard elliptic estimate and Gagliardo-Nirenberg-Sobolev inequality, we can estimate $J_1$ by:
\begin{align*}
    J_1 &\le \|\Delta\rho\|_{L^2}\|\nabla\rho\|_{L^3}\|\nabla(-\Delta)^{-1}\rho\|_{L^6} \lesssim \|\Delta\rho\|_{L^2}\|\nabla\rho\|_{L^3}\|\nabla(-\Delta)^{-1}\rho\|_1\\
    &\lesssim \|\Delta\rho\|_{L^2}\|\nabla\rho\|_{L^2}^{1/2}\|\nabla\rho\|_{1}^{1/2}\|\rho\|_{L^2} \lesssim \|\rho\|_{2}^{3/2}\|\nabla\rho\|_{L^2}^{1/2}\|\rho\|_{L^2}\\
    &\le \epsilon\|\Delta\rho\|_{L^2}^2 + C(\epsilon)\|\nabla\rho\|_{L^2}^2\|\rho\|_{L^2}^4,
\end{align*}
where we also used Young's inequality in the final step.

%Using the Sobolev interpolation inequality
%$
%\|\rho\|_{3/2} \le \|\rho\|_1^{1/2}\|\rho\|_2^{1/2},
%$
We are going to use the following Gagliardo-Nirenberg inequalities: in dimension three,
\[ \|\rho\|_{L^4} \lesssim \|\Delta \rho\|_{L^2}^{3/8}\|\rho\|_{L^2}^{5/8}; \,\,\,
\|\rho\|_{L^4} \lesssim \|\rho\|_{1}^{3/4}\|\rho\|_{L^2}^{1/4}. \]
Then we can estimate $J_2$ as follows:
\begin{align*}
J_2 &\le \|\Delta\rho\|_{L^2}\|\rho\|_{L^4}^2 \le C\|\Delta\rho\|_{L^2}\|\Delta \rho\|^{1/2}_{L^2}\|\rho\|_{L^2}^{5/6}
\|\rho\|_1^{1/2}\|\rho\|_{L^2}^{1/6} \\
& = C\|\Delta\rho\|_{L^2}^{3/2}\|\rho\|_1^{1/2}\|\rho\|_{L^2} \le \epsilon\|\Delta\rho\|_{L^2}^2 + C(\epsilon)\|\nabla\rho\|_{L^2}^2\|\rho\|_{L^2}^4,
\end{align*}
Collecting the estimates above and choosing $\epsilon$ to be sufficiently small, we obtain the following:
\begin{align}
    \label{H1rho}
    \frac{d}{dt}\|\nabla \rho\|_{L^2}^2 + \|\Delta \rho\|_{L^2}^2 &\lesssim \left(\|\rho\|_{L^2}^4 + \|u\|_{2}^2\right)\|\nabla \rho\|_{L^2}^2
\end{align}
On the other hand, we test \eqref{eq:stokes2} by $\calA u$. Integrating by parts, we have
$$
\frac{1}{2}\frac{d}{dt}\|\nabla u\|_{L^2}^2 + \|\calA u\|_{L^2}^2 = g\int_\Omega \calA u \cdot \rho  e_z \le \frac{1}{2}\|\calA u\|_{L^2}^2 + \frac{g^2}{2}\|\rho\|_{L^2}^2
$$
Rearranging the above and using Theorem \ref{stokesest}, we conclude that,
\begin{equation}
    \label{H1u}
    \frac{d}{dt}\|\nabla u\|_{L^2}^2 + \|u\|_{2}^2 \le g^2\|\rho\|_{L^2}^2 \le 4g^2\|\rho_0\|_{L^2}^2,\; t \in [0,T_*],
\end{equation}
where the last inequality is due to Proposition \ref{prop:l2bd}. Integrating \eqref{H1u} from $0$ to $t$, $t \le T_*$ and then taking supremum of $t$ on $[0,T_*]$, we obtain
$$
\sup_{t \in [0,T_*]}\|\nabla u(t)\|_{L^2}^2 \le 4g^2\|\rho_0\|_{L^2}^2T_* + \|\nabla u_0\|_{L^2}^2;
$$
in addition,
\begin{align}\label{aux828b} \int_0^{T_*}\|u(t)\|_{2}^2 \leq 4g^2\|\rho_0\|^2_{L^2}T_* +\|\nabla u_0\|_{L^2}^2.   \end{align}
%Plugging the bound above into the time-integrated \eqref{H1u}, it is straightforward to obtain the bound
It follows that
$$
\sup_{t \in [0,T_*]}\|u(t)\|_1^2 + \int_0^{T_*}\|u(t)\|_{2}^2 dt\le C(u_0,\rho_0, g).
$$

Integrating \eqref{H1rho} and using \eqref{aux828b}, we have that for all $t \in [0,T_*]$,
\begin{align*}
\|\nabla \rho(t)\|_{L^2}^2 &\lesssim \| \rho_0\|_{1}^2\exp\left(\int_0^{T_*}( \|\rho\|_{L^2}^4 + \|u\|_{2}^2 ds\right) \le \|\rho_0\|_{1}^2\exp\left(C(\rho_0,g)T_* + \|u_0\|_1^2 \right) < \infty.
\end{align*}
Similarly to the case of $u,$ we can also use \eqref{H1rho} to control $\int_0^{T_*}\|\rho(t)\|_{2}^2 dt$ as well, arriving at
$$
\sup_{t \in [0,T_*]}\|\rho(t)\|_1^2 + \int_0^{T_*}\|\rho(t)\|_{2}^2 \le C(u_0,\rho_0, g).
$$
We have thus showed the spatial regularity of $\rho$ and $u$.

Finally, we shall obtain regularity estimates for the time derivatives. Using the equation \eqref{eq:ksstokes}, we see that
$$
\p_t\rho = -u\cdot \nabla\rho + \Delta\rho - \divv(\rho\nabla(-\Delta)^{-1}\rho)\,\,\,{\rm and}\,\,\, \p_t u = -\calA u + g\bP(\rho e_z).
$$
Using standard Sobolev embeddings and elliptic estimate, we have the following bounds:
\begin{align*}
    \int_0^{T_*}\|u\cdot \nabla\rho(t)\|_{L^2}^2 dt &\le \int_0^{T_*}\|u\|_{L^6}^2\|\nabla\rho\|_{L^3}^2 dt\lesssim \sup_{t \in [0,T_*]}\|u(t)\|_{1}^2 \int_0^{T_*}\|\rho(t)\|_{2}^2dt,\\
    \int_0^{T_*}\|\Delta\rho\|_{L^2}^2 dt &\le \int_0^{T_*}\|\rho(t)\|_2^2dt,\\
    \int_0^{T_*}\|\divv(\rho\nabla(-\Delta)^{-1}\rho)\|_{L^2}^2 dt &\lesssim \int_0^{T_*}\|\rho\|_{L^4}^4 + \|\nabla\rho \cdot \nabla(-\Delta)^{-1}\rho\|_{L^2}^2 dt\\
    &\lesssim \sup_{t\in[0, T_*]}\|\rho(t)\|_{1}^4 T_* + \sup_{t\in[0, T_*]}\|\rho(t)\|_{1}^2\int_0^{T_*}\|\rho(t)\|_{2}^2 dt,\\
    \int_0^{T_*} \|\calA u\|_{L^2}^2 + g\|\bP \rho\|_{L^2}^2 dt &\le \int_0^{T_*} \|u\|_2^2 + g\|\rho\|_{L^2}^2 dt.
\end{align*}
The above estimates and bounds we proved earlier imply that
\[ \int_0^{T_*}\|\p_t\rho\|_{L^2}^2 dt + \int_0^{T_*}\|\p_tu\|_{L^2}^2 dt \le C(u_0,\rho_0, g),\] and the proof is thus complete.
\end{proof}

With the regularity estimates above, we may construct solutions $(\rho, u)$ from $(\rhon, \un)$. The following standard compactness theorem is useful. We refer interested readers to Theorem IV.5.11 in \cite{BF} and Theorem 4 of Chapter 5 in \cite{Evans} for related proofs.
\begin{thm}\label{thm:compact}
Let
\begin{align*}
E_1 := \{\rho \in L^2((0,T);H^2),\; \p_t\rho \in L^2((0,T); L^2)\},\\
E_2 := \{u \in L^2((0,T);H^2 \cap V),\; \p_t u \in L^2((0,T); H)\}
\end{align*}
for some $T > 0$. Then $E_1$ is continuously embedded in $C([0,T], H^1)$, and $E_2$ is continuously embedded in $C([0,T], V)$.
\end{thm}
\begin{cor}
\label{cor:lowreg}
Given initial data $\rho_0 \in H^1_0$, $u \in V$, there exists a weak solution $(\rho,u)$ of the system \eqref{eq:ksstokes} satisfying
\begin{align}
\rho \in C([0,T]; H^1_0) \cap L^2((0,T); H^2\cap H^1_0),&\; u \in C([0,T]; V) \cap L^2((0,T); H^2\cap V),\label{regularity1}\\
\p_t \rho \in C([0,T_*];H^{-1}_0),&\; \p_t u \in C([0,T_*]; V^*).\label{regularity2}
\end{align}
\end{cor}
\begin{proof}
%Keeping Remark \ref{rmk:apriori} in mind,
The uniform bounds in Proposition \ref{prop:h1bd} inform us that there exists a subsequence of $\{\rhon\}_n, \{\un\}_n$, which we still denote by $\rhon, \un$, and $\rho, u$, such that
\begin{enumerate}
\item $\rhon \rightharpoonup \rho$ weak-$\ast$ in $L^\infty((0,T_*); H^1_0)$, weakly in $L^2((0,T); H^2\cap H^1_0)$; $\p_t \rhon \rightharpoonup \p_t \rho$ weakly in $L^2((0,T); L^2)$,
\item $\un \rightharpoonup u$ weak-$\ast$ in $L^\infty((0,T_*); V)$, weakly in $L^2((0,T); H^2\cap V)$; $\p_t \un \rightharpoonup \p_t u$ weakly in $L^2((0,T); H)$.
\end{enumerate}
It is straightforward to check that the limits $\rho$ and $u$ satisfy \eqref{eq:ksstokes} in the sense of distribution. Evoking Theorem \ref{thm:compact}, we have proved \eqref{regularity1}.

Now, we show $\p_t u \in C([0,T_*]; V^*)$. In view of \eqref{eq:stokes2}, it suffices to show that $-\calA u + g\rho e_2 \in C([0,T_*]; V^*)$. For simplicity, we show that the most singular term $\calA u \in C([0,T_*]; V^*)$, and the argument for $g\rho e_2$ follows similarly.  Choose $t, s \in [0,T_*]$ and pick arbitrary vector field $\phi \in V$. Integrating by parts, we observe that
\begin{align*}
\int_\Omega (\calA u(t,x) - \calA u(s, x)) \cdot \phi(x) dx &= \int_\Omega \calA^{1/2}(u(t,x) - u(s,x))\cdot \calA^{1/2}\phi dx\\
&\le \|u(t,\cdot) - u(s,\cdot)\|_1\|\phi\|_1.
\end{align*}
By duality, we observe that
$$
\|\calA u(t,\cdot) - \calA u(s, \cdot)\|_{V^*} \le \|u(t,\cdot) - u(s,\cdot)\|_1 \to 0
$$
as $t \to s$ due to $u \in C([0,T_*]; V)$. Thus, we have showed that $\calA u \in C([0,T_*]; V^*)$ and hence $\p_t u \in C([0,T_*]; V^*)$.

To show the needed regularity of $\p_t \rho$, it suffices to show that $-u\cdot \nabla \rho + \Delta \rho - \divv(\rho\nabla(-\Delta)^{-1}\rho) \in C([0,T_*]; H^{-1}_0)$. Similarly, we prove strong continuity for the most singular term $u \cdot \nabla \rho$. The rest of the terms will follow from a similar argument. Let $t, s \in [0,T_*]$. Picking $\varphi \in H^1_0$ and integrating by parts, we have
\begin{align*}
\int_\Omega &(u(t,x)\cdot \nabla \rho(t,x) - u(s,x)\cdot \nabla \rho(s,x))\varphi(x) dx\\
&= \int_\Omega (u(t,x) - u(s,x))\cdot\nabla \rho(t,x)\varphi dx + \int_\Omega u(s,\cdot)\cdot\nabla (\rho(t,x) - \rho(s,x))\varphi(x) dx\\
&= \int_\Omega \divv((u(t,x) - u(s,x))\rho(t,x))\varphi dx + \int_\Omega \divv(u(s,\cdot) (\rho(t,x) - \rho(s,x)))\varphi(x) dx\\
&= -\int_\Omega ((u(t,x) - u(s,x))\rho(t,x))\cdot\nabla\varphi dx - \int_\Omega (u(s,\cdot) (\rho(t,x) - \rho(s,x)))\cdot\nabla\varphi(x) dx.
\end{align*}
The first term on RHS can be estimated by:
\begin{align*}
\int_\Omega ((u(t,x) - u(s,x))\rho(t,x))\cdot\nabla\varphi dx &\le \|u(t,\cdot) - u(s,\cdot)\|_{L^3}\|\rho(t,\cdot)\|_{L^6}\|\varphi\|_1\\
&\lesssim \|u(t,\cdot) - u(s,\cdot)\|_1\|\rho(t,\cdot)\|_1\|\varphi\|_1\\
&\le C(\rho_0,u_0,g)\|u(t,\cdot) - u(s,\cdot)\|_1\|\varphi\|_1.
\end{align*}
Note that we used Sobolev embedding in the second inequality and the uniform bound of $\rho$ in $ L^\infty((0,T_*); H^1_0)$ norm in the last inequality. Similarly, we can estimate the second term on RHS by:
$$
\int_\Omega (u(s,\cdot) (\rho(t,x) - \rho(s,x)))\cdot\nabla\varphi(x) dx \le C(\rho_0,u_0,g)\|\rho(t,\cdot) - \rho(s,\cdot)\|_1\|\varphi\|_1
$$
thanks to $u \in L^\infty((0,T); V)$. Combining the two estimates above and using duality, we conclude that
$$
\|u(t,\cdot)\cdot\nabla\rho(t,\cdot)-u(s,\cdot)\cdot\nabla\rho(s,\cdot)\|_{H^{-1}_0} \le C(\rho_0,u_0,g)(\|u(t,\cdot) - u(s,\cdot)\|_1 + \|\rho(t,\cdot) - \rho(s,\cdot)\|_1) \to 0
$$
as $t \to s$ due to $u \in C([0,T_*]; V)$ and $\rho \in C([0,T_*];H^1_0)$. This verifies $\p_t \rho \in C([0,T_*];H^{-1}_0)$, and we have proved \eqref{regularity2}.
\end{proof}

\subsection{Higher Order \textit{a priori} Estimates}\label{subsect:high}
Our next task is to establish the smoothness of a solution $(\rho,u)$ for positive times, namely
$$
\rho \in C^\infty((0,T_*] \times \Omega),\; u \in C^\infty((0,T_*] \times \Omega),
$$
via energy estimates in higher order Sobolev norms. We would like to remark on the following caveat: with Dirichlet boundary condition imposed on both $\rho$ and $u$, one cannot obtain higher order Sobolev estimates by commuting the differential operator $\p^s$ with the equation, where $\p^s$ denotes a general $s$-th order spatial derivative. The main reason is that when we treat the dissipation term, integration by parts incurs a boundary term that is difficult to control.
% Namely,
%$$
%\int_\Omega \p^s \rho \Delta \p^s \rho = -\int_\Omega |\nabla \p^s\rho|^2 + \int_{\p \Omega} \p^s\rho \frac{\p}{\p n}\p^s \rho,
%$$
%where $\p/\p n$ denotes the normal derivative.
To remedy this issue, we commute time derivatives $\p_t^k$ through the equation. It is clear that no boundary terms are generated since $\p_t$ preserves Dirichlet boundary condition. By applying this strategy, we can improve regularity in time, after which spatial regularity can be upgraded using elliptic estimates.

Again, to obtain the claimed regularity we should proceed by the Galerkin scheme and perform the estimates in Proposition \ref{prop:highreg} for the approximated solutions. Since this step is similar to that in Corollary \ref{cor:lowreg}, we omit this tedious part and will proceed with
only \textit{a priori} estimates as follows.

\begin{prop}
\label{prop:highreg}
Assume $(\rho,u)$ is a regular solution to problem \eqref{eq:ksstokes} with initial condition $\rho_0 \in H_0^1, u_0 \in V$. Then the following bounds hold:
\begin{align}
    t^{k}\left(\|\p_t^l\rho(t,\cdot)\|_{1+k-2l}^2 + \|\p_t^l u(t,\cdot)\|_{1+k-2l}^2\right) &\le C(\rho_0,u_0,g,k),\label{est:regest1}\\
    t^{k}\int_t^{T_*}\left(\|\p_t^l\rho(\tau,\cdot)\|_{2+k-2l}^2 + \|\p_t^l u(\tau,\cdot)\|_{2+k-2l}^2\right)d\tau &\le C(\rho_0, u_0,g,k) \label{est:regest2},
\end{align}
for any $t \in (0,T_*]$, $k \in \N$, $0 \le l \le \lfloor \frac{k+1}{2}\rfloor,$ where $\lfloor\cdot\rfloor$ denotes the floor function.
\end{prop}
\begin{proof}
    We prove the proposition by inducting on $k$. Since $k=0$ case is already proved by Proposition \ref{prop:h1bd}, we now assume that the statement holds up to index $k-1$. We will discuss two cases based on the parity of $k$. We also remind the readers that the constant $C(\rho_0,u_0,g,k)$ might change from line to line.
    \begin{enumerate}
    \item \textbf{$k$ is odd.} Let us write $S = \frac{k+1}{2}$, and define the $s$-energy
    \[ E_s(\tau) = \|\srho{s}(\tau,\cdot)\|_{L^2}^2 + \|\su{s}(\tau,\cdot)\|_{L^2}^2 \] for any $0\le s \le S$. From now on, we fix arbitrary $t \in (0,T_*]$. This case can be detailed into the following steps.

    \textbf{Step 1: show \eqref{est:regest1}, \eqref{est:regest2} with $l = S$.}
    Commuting $\p_t^s$ with \eqref{eq:ksstokes} for $0 \le s \le S$, we obtain that
    \begin{subequations}
    \label{seqn}
    \begin{align}
        \p_t \srho{s} &- \Delta \srho{s} + \sum_{r=0}^s{s \choose r}\bigg[(\su{r}\cdot\nabla)\srho{s-r}+ \srho{s-r}\srho{r} + \nabla\srho{s-r} \cdot \nabla (-\Delta)^{-1}(\srho{r})\bigg] = 0, \label{srhoeqn}\\
        \p_t \su{s} &+ \calA\su{s} = g\bP(\srho{s}e_z), \label{sueqn}
    \end{align}
    \end{subequations}
    equipped with boundary conditions $\srho{s}|_{\p\Omega} = 0$, $\su{s}|_{\p\Omega} = 0$. Testing \eqref{sueqn} with $s=S$ by $\su{S}$, we obtain that
    $$
    \frac{1}{2}\frac{d}{dt}\|\su{S}\|_{L^2}^2 + \|\nabla \su{S}\|_{L^2}^2 \le \frac{g}{2}\left(\|\su{S}\|_{L^2}^2 + \|\srho{S}\|_{L^2}^2\right).
    $$
    Testing \eqref{srhoeqn} with $s=S$ by $\srho{S}$:
    \begin{equation}
        \frac{1}{2}\frac{d}{dt}\|\srho{S}\|_{L^2}^2 + \|\nabla \srho{S}\|_{L^2}^2 = \sum_{r=0}^S{S \choose r}(I_r + J_r + K_r),
    \end{equation}
    where
    $$
    I_r = \int_\Omega (\srho{S})(\su{r}\cdot\nabla)\srho{S-r},\;J_r = \int_\Omega (\srho{S})\srho{S-r}(\srho{r}),
    $$
    $$
    K_r = \int_\Omega (\srho{S})\nabla\srho{S-r} \cdot \nabla (-\Delta)^{-1}(\srho{r}).
    $$

To estimate $I_r$, first note that $I_0 = 0$ by incompressibility and integration by parts.
    For $1 \le r \le S-1$, we integrate $I_r$ by parts once to obtain:
    $$
    I_r = -\int_\Omega \p_j\srho{S}\su{r}_j\srho{S-r},
    $$
    where we also used the incompressibility of $\su{r}$. Thus, we can estimate:
    \begin{align*}
        I_r &\le \|\nabla \srho{S}\|_{L^2}\|\su{r}\|_{L^3}\|\srho{S-r}\|_{L^6} \le \delta\|\nabla\srho{S}\|_{L^2}^2 + C(\delta)\|\su{r}\|_{L^3}^2\|\srho{S-r}\|_1^2,
    \end{align*}
    for some $\delta > 0$. If $r = S$, we instead estimate:
    $$
    I_S \le \|\nabla \srho{S}\|_{L^2}\|\su{S}\|_{L^2}\|\rho\|_{L^\infty} \le \delta\|\nabla\srho{S}\|_{L^2}^2 + C(\delta)\|\rho\|_2^2\|\su{S}\|_{L^2}^2.
    $$
    This concludes the estimates of $I_r$. To estimate $J_r$, we note that if $r = 0$ or $r= S$, we have
    \begin{align*}
        J_r &\le \|\srho{S}\|_{L^2}^2\|\rho\|_{L^\infty} \lesssim \|\srho{S}\|_{L^2}^2\|\rho\|_2
    \end{align*}
    If $1 \le r \le S-1$, then we have
    \begin{align*}
        J_r \le \frac{1}{2}\|\srho{S}\|_{L^2}^2 + \frac12\|\srho{r}\|_1^2\|\srho{S-r}\|_1^2.
    \end{align*}

    Now we estimate $K_r$. If $r = 0$, we use the standard elliptic estimate and Young's inequality to obtain:
    $$
    K_0 \le \delta\|\nabla \srho{S}\|_{L^2}^2 + C(\delta)\|\rho\|_{1}^2 \|\srho{S}\|_{L^2}^2,
    $$
    where $\delta > 0$. If $r = S$, we apply elliptic estimate and Sobolev embedding:
    $$
    K_S \le \|\nabla \rho\|_{L^3}\|\srho{S}\|_{L^2}\|\nabla(-\Delta)^{-1}\srho{S}\|_{L^6} \lesssim \|\nabla \rho\|_{1}\|\srho{S}\|_{L^2}^2.
    $$
    If $1 \le r \le S-1$, we can estimate
    \begin{align*}
    K_r &\le \|\nabla \srho{S-r}\|_{L^3}\|\srho{S}\|_{L^2}\|\nabla(-\Delta)^{-1}\srho{r}\|_{L^6} \le \frac{1}{2}\|\srho{S}\|_{L^2}^2 + C \|\nabla \srho{S-r}\|_1^2\|\srho{r}\|_{L^2}^2.
    \end{align*}
After choosing $\delta>0$ to be sufficiently small, the above estimates yield the following differential inequality: for $\tau \in (0,T_*)$,
    \begin{align}
        \label{Senergyodd}
        \frac{dE_S}{d\tau} + \|\nabla\srho{S}(\tau,\cdot)\|_{L^2}^2 + &\|\nabla\su{S}(\tau,\cdot)\|_{L^2}^2 \le C(k)\bigg[ \left(1+g+\|\rho\|_2 + \|\rho\|_2^2 \right) E_S(\tau)\notag\\
        &+\sum_{r=1}^{S-1}\big(\|\nabla \srho{S-r}\|_1^2\|\srho{r}\|_{L^2}^2+ \|\srho{r}\|_1^2\|\srho{S-r}\|_1^2+\|\su{r}\|_{L^3}^2\|\srho{S-r}\|_1^2\big)\bigg]\notag\\
        & = C(k)\left(F(\tau)E_S(\tau) + \sum_{r=1}^{S-1}G_r(\tau)\right)
    \end{align}
with \begin{gather*} F(\tau) = 1+g+\|\rho(\tau,\cdot)\|_2 + \|\rho(\tau,\cdot)\|_2^2, \\
 G_r(\tau) = \|\nabla \srho{S-r}\|_1^2\|\srho{r}\|_{L^2}^2+ \|\srho{r}\|_1^2\|\srho{S-r}\|_1^2+\|\su{r}\|_{L^3}^2\|\srho{S-r}\|_1^2. \end{gather*}
    To proceed, we need the following useful lemma:
\begin{lem}
\label{lem1}
There exists $\tau_0 \in [t/2,t]$ such that $E_S(\tau_0) \le C(\rho_0,u_0,g,k)t^{-k}$.
\begin{proof}
Let us consider \eqref{seqn} with $s = S-1$. For any $\tau \in [t/2,t]$, we note that by \eqref{sueqn},
\begin{align*}
\|\su{S}(\tau)\|_{L^2}^2 &\lesssim \|\calA\su{S-1}\|_{L^2}^2 + g^2\|\srho{S-1}\|_{L^2}^2 \le \|\su{S-1}\|_2^2 + g^2\|\srho{S-1}\|_{L^2}^2.
\end{align*}
Integrating over $[t/2,t]$ and using \eqref{est:regest2} at index $k-1$ (which is valid as this is part of the induction hypothesis), we obtain
\begin{equation}
\label{lem1ueq}
\int_{t/2}^t \|\su{S}(\tau)\|_{L^2}^2 d\tau \le \int_{t/2}^{T_*} \|\su{S}(\tau)\|_{L^2}^2 d\tau\le C(\rho_0,u_0,g,k)t^{1-k}.
\end{equation}

Similarly, applying H\"older inequality to \eqref{srhoeqn}, we have
    \begin{align}
    \label{srhointer}
        \|\srho{S}\|_{L^2}^2 &\lesssim \|\srho{S-1}\|_{2}^2 + \sum_{r=0}^{S-1}C(k)\bigg(\|\su{r}\|_1^2\|\nabla\srho{S-1-r}\|_1^2\notag\\
        &+ \|\srho{S-1-r}\|_1^2\|\srho{r}\|_1^2 + \|\nabla \srho{S-1-r}\|_1^2\|\srho{r}\|_{L^2}^2\bigg).
    \end{align}

    Observe that given the induction hypothesis, applying \eqref{est:regest2} with index $k-1$, we have
    $$
    \int_{t/2}^t \|\srho{S-1}(\tau)\|_{2}^2d\tau \le C(\rho_0,u_0,g,k)t^{1-k}.
    $$
    Also, for $r = 0,\hdots, S-1$,
    $$
    \int_{t/2}^t \|\su{r}(\tau)\|_1^2\|\nabla\srho{S-1-r}(\tau)\|_1^2 d\tau \le C(\rho_0,u_0,g,k)t^{-2r}t^{-2(S-r-1)} = C(\rho_0,u_0,g,k)t^{1-k},
    $$
    where we applied \eqref{est:regest1} with index $2r$ to $\|\su{r}(\tau)\|_1$ and \eqref{est:regest2} with index $2(S-r-1)$ to $\|\nabla\srho{S-1-r}(\tau)\|_1$. In a similar fashion, we can also obtain the following bound:
    $$
    \int_{t/2}^t \left[\|\srho{S-1-r}(\tau)\|_1^2\|\srho{r}(\tau)\|_1^2 + \|\nabla \srho{S-1-r}(\tau)\|_1^2\|\srho{r}(\tau)\|_{L^2}^2\right]d\tau \le C(\rho_0,u_0,g,k)t^{1-k}.
    $$
    Collecting the estimates above and combining with \eqref{srhointer}, we have
    \begin{equation}
    \label{lem1rhoeq}
    \int_{t/2}^t \|\srho{S}(\tau)\|_{L^2}^2d\tau \le C(\rho_0,u_0,g,k)t^{1-k}.
    \end{equation}
    Combining \eqref{lem1ueq} and \eqref{lem1rhoeq}, we have
    $$
    \int_{t/2}^t \left( \|\su{S}(\tau)\|_{L^2}^2+\|\srho{S}(\tau)\|_{L^2}^2 \right) d\tau \le C(\rho_0,u_0,g,k)t^{1-k}.
    $$
    By mean value theorem, we can find a $\tau_0 \in (t/2,t)$ such that
    $$
    E_S(\tau_0) = \|\su{S}(\tau_0)\|_{L^2}^2 + \|\srho{S}(\tau_0)\|_{L^2}^2 \le C(\rho_0,u_0,g,k)t^{-k},
    $$
    and this concludes the proof.
\end{proof}
\end{lem}

We also need another lemma that treats the terms $G_r$.
    \begin{lem}
    \label{lem2}
    Let $\tau_0$ be chosen as in Lemma \ref{lem1}. Then for any $r = 1,\hdots, S-1$, we have
    $$
    \int_{\tau_0}^{T_*}G_r(\tau)d\tau \le C(\rho_0, u_0, g, k)t^{-k}.
    $$
    \begin{proof}
    We fix $r = 1,\hdots, S-1$. By definition of $G_r$, we can write
    \begin{align*}
     \int_{\tau_0}^{T_*}G_r(\tau)d\tau &= \int_{\tau_0}^{T_*}\left( \|\nabla \srho{S-r}\|_1^2\|\srho{r}\|_{L^2}^2 + \|\srho{r}\|_1^2\|\srho{S-r}\|_1^2+\|\su{r}\|_{L^3}^2\|\srho{S-r}\|_{1}^2 \right) d\tau\\
     &=: \int_{\tau_0}^{T_*}\left( G_r^1(\tau) + G_r^2(\tau)+G_r^3(\tau) \right) d\tau.
    \end{align*}
    Applying \eqref{est:regest1} with index $2r-1$ and \eqref{est:regest2} with index $k-2r+1$ to terms $\|\srho{r}\|_{L^2}^2$ and $\|\nabla \srho{S-r}\|_1^2$ respectively, we observe that
    \begin{align*}
    \int_{\tau_0}^{T_*}G_r^1(\tau)d\tau &\le C(\rho_0,u_0,g,k)\tau_0^{1-2r}\int_{\tau_0}^{T_*}\| \srho{S-r}(\tau)\|_2^2d\tau\\
    &\le C(\rho_0,u_0,g,k)\tau_0^{-(2r-1)}\tau_0^{-(k-2r+1)}\\
    &\le C(\rho_0,u_0,g,k)t^{-k},
    \end{align*}
    where we used the fact that $\tau_0 > t/2$.

    To study the term involving $G_r^2$, we will apply \eqref{est:regest1} with index $2r$ and \eqref{est:regest2} with index $k-2r$ to terms $\|\srho{r}\|_1^2$ and $\|\srho{S-r}\|_1^2$ respectively. This yields:
    \begin{align*}
        \int_{\tau_0}^{T_*}G_r^2(\tau)d\tau &\le C(\rho_0,u_0,g,k)\tau_0^{-2r}\int_{\tau_0}^{T_*}\| \srho{S-r}(\tau)\|_1^2d\tau\\
        &\le C(\rho_0,u_0,g,k)\tau_0^{-2r}\tau_0^{-(k-2r)}\\
    &\le C(\rho_0,u_0,g,k)t^{-k},
    \end{align*}

    Finally, using Sobolev embedding, Gagliardo-Nirenberg-Sobolev inequality, and Cauchy-Schwarz inequality,
    \begin{align*}
        \int_{\tau_0}^{T_*}G_r^3(\tau)d\tau &\le \int_{\tau_0}^{T_*}\|\su{r}\|_{L^3}^2\|\srho{S-r}\|_1^2d\tau \lesssim \int_{\tau_0}^{T_*}\|\su{r}\|_{L^2}\|\nabla\su{r}\|_{L^2}\|\srho{S-r}\|_1^2 d\tau\\
        &\le C(\rho_0,u_0,g,k)\tau_0^{-\frac{2r-1}{2}}\tau_0^{-r}\int_{\tau_0}^{T_*} \|\srho{S-r}\|_1^2 d\tau\\
     %   &\le C(\rho_0,u_0,g,k)\tau_0^{-\frac{2r-1}{2}}\tau_0^{-\frac{k-2r+1}{2}}\tau_0^{-\frac{2r-1}{2}}\tau_0^{-\frac{k-2r}{2}}\\
        &\le C(\rho_0,u_0,g,k)t^{-(k-\frac{1}{2})} \le C(\rho_0,u_0,g,k)t^{-k}.
    \end{align*}
    Note that we applied \eqref{est:regest1} with index $2r-1$ to $\|\su{r}\|_{L^2}$, \eqref{est:regest1} with index $2r$ to $\|\nabla \su{r}\|_{L^2}$,
    %\eqref{est:regest2} with index $2r-1$ to $\|\nabla\su{r}\|_{L^2}$, \eqref{est:regest1} with index $k-2r+1$ to $\|\srho{S-r}\|_1$,
    and \eqref{est:regest2} with index $k-2r$ to the other $\|\srho{S-r}\|_1^2$. We also used $\tau_0 \le T_* \le 1$ in the final inequality. The proof is thus completed after we combine the estimates above.
    \end{proof}
    \end{lem}

    Using induction hypothesis at $k = 0$, we have $F \in L^1(0,T_*)$ with the bound $\|F\|_{L^1(0,T_*)} \le C(u_0,\rho_0, g)$. We may thus apply Gr\"onwall inequality to \eqref{Senergyodd} on time interval $[\tau_0, t]$, where $\tau_0$ is selected as in Lemma \ref{lem1} above. Using the two lemmas above, we have
    \begin{align}
    E_S(t) &\le C(k)\left(E_S(\tau_0) + \sum_{r=1}^{S-1}\int_{\tau_0}^{t}G_r(\tau)d\tau\right)\exp\left(\|F\|_{L^1(0,T_*)}\right)\notag\\
    &\le C(\rho_0,u_0,g,k)t^{-k}, \label{regest1S}
    \end{align}
    where we recall that $T_*$ depends only on $\rho_0$. This verifies \eqref{est:regest1}. To verify \eqref{est:regest2}, we integrate \eqref{Senergyodd} on interval $[t, T_*]$, which yields:
    $$
        \int_t^{T_*}\left(\|\nabla\srho{S}(\tau)\|_{L^2}^2 + \|\nabla\su{S}(\tau)\|_{L^2}^2\right)d\tau \le E_S(t) +C(k)\bigg(\int_t^{T_*}F(\tau)E_S(\tau)d\tau + \sum_{r=1}^{S-1}\int_t^{T_*}G_r(\tau)d\tau\bigg).
    $$
    Using \eqref{regest1S}, Lemma \ref{lem2}, and the fact that $\frac{t}{2} <\tau_0 < t$, we can estimate the above by:
    \begin{align*}
        \int_t^{T_*}(\|\nabla\srho{S}(\tau)\|_{L^2}^2 &+ \|\nabla\su{S}(\tau)\|_{L^2}^2)d\tau \le C(\rho_0,u_0,g,k)t^{-k} \\
        &+ C(\rho_0,u_0,g,k)(t^{-k}\|F\|_{L^1} + t^{-k})\\
        &\le C(\rho_0,u_0,g,k)t^{-k}.
    \end{align*}
    This concludes the proof of \eqref{est:regest2} with $l = S$.

    \textbf{Step 2: show \eqref{est:regest1}, \eqref{est:regest2} with $l < S$.}
     We will show how we obtain the case when $l = S-1$. Then the rest just follows from another induction on $l = 1,\hdots, S$ backwards.

    We may rewrite the equations \eqref{seqn} with $s = S-1$ as
    \begin{subequations}
    \label{seqn2}
    \begin{align}
        - \Delta \srho{S-1} &= -\srho{S} - \sum_{r=0}^{S-1}{S-1 \choose r}\bigg[\su{r}\cdot\nabla \srho{S-1-r}\notag + \srho{S-1-r}\srho{r} + \nabla\srho{S-1-r} \cdot \nabla (-\Delta)^{-1}(\srho{r})\bigg]\notag\\
        &= -\srho{S} + R_1 \label{srhoeqn2}\\
        \calA\su{S-1} &= -\su{S} + g\bP(\srho{S-1}e_z) = -\su{S} + R_2 \label{sueqn2}
    \end{align}
    \end{subequations}
    Here, $R_1,R_2$ are the remainder terms which are essentially of lower order. We will see that these terms can be treated by the induction hypothesis on $k$. To illustrate this, we show that the following estimates hold:
    \begin{lem}
    \label{lem3}
    For any $t \in (0,T_*]$,
    $$
    t^{k-\frac{1}{4}}(\|R_1(t)\|_{L^2}^2 + \|R_2(t)\|_{L^2}^2) \le C(\rho_0, u_0, g,k),
    $$
    $$
    t^{k-\frac{1}{4}}\int_{t}^{T_*}\left(\|R_1(\tau)\|_1^2 + \|R_2(\tau)\|_1^2\right)d\tau \le C(\rho_0, u_0, g,k).
    $$
    \begin{proof}
    First, it is straightforward to obtain the following bounds for $R_2$ by directly imposing the induction hypothesis at index $k-1$:
    \begin{equation}
    \label{est:lem3aux1}
    t^{k-1}\|R_2(t)\|_{L^2}^2 + t^{k-1}\int_t^{T_*}\|R_2(t)\|_{1}^2 dt \le C(\rho_0,u_0,g,k).
    \end{equation}
    Prior to estimating $R_1$, we first need an improved bound for $\|u\|_2$: invoking \eqref{est:lem3aux1} with $k=1$, we have
    $$
    \|R_2(t)\|_{L^2}^2 \le C(\rho_0, u_0, g).
    $$
    Since $S = 1$ when $k = 1$ by definition, we apply the Stokes estimate to \eqref{sueqn2} with $S = 1$ to see that
    \begin{equation}
    \label{est:lem3uh2}
    \|u\|_2^2 \lesssim \|\p_t u\|_{L^2}^2 + \|R_2\|_{L^2}^2 \le C(\rho_0, u_0, g)(t^{-1} + 1) \le C(\rho_0, u_0, g)t^{-1},
    \end{equation}
    where we used Step 1 with $k = 1$ above.
     Now, we are ready to estimate $R_1$. We first note that it involves $3$ typical terms, namely

    $$
    R_{11}^r:=\su{r}\cdot\nabla\srho{S-1-r},\;R_{12}^r:=\srho{S-1-r}\srho{r},\;R_{13}^r:=\nabla\srho{S-1-r} \cdot \nabla (-\Delta)^{-1}(\srho{r}),
    $$

    where $0\le r \le S-1$. We will prove suitable bounds for $R_{11}^r$, and the rest can be bounded more easily since these terms involve fewer derivatives.
    If $1 \le r \le S-1$, then by H\"older inequality:
    \begin{align*}
    \|R_{11}^r\|_{L^2}^2 &\le \|\su{r}\|_{L^6}^2\|\nabla\srho{S-1-r}\|_{L^3}^2 \lesssim \|\su{r}\|_1^2\|\srho{S-1-r}\|_1\|\srho{S-1-r}\|_2\\
    &\le C(\rho_0, u_0, g,k)t^{-2r}t^{-\frac{k-2r-1}{2}}t^{-\frac{k-2r}{2}}\\
    &\le C(\rho_0, u_0, g,k)t^{-(k-\frac12)},
    \end{align*}
    where we used \eqref{est:regest1} at indices $2r, k-2r-1,k-2r$ respectively.\\
    If $r = 0$, then we observe that $R_{11}^0 = u\cdot \nabla\srho{S-1}$. We estimate as follows:
    \begin{align*}
    \|R_{11}^0\|_{L^2}^2 &\le \|u\|_{L^\infty}^2\|\srho{S-1}\|_1^2 \le \|u\|_{L^2}^{1/2}\|u\|_{2}^{3/2}\|\srho{S-1}\|_1^2 \\
    &\le C(\rho_0,u_0,g,k)t^{-3/4}t^{-(k-1)} = C(\rho_0,u_0,g,k)t^{-(k-1/4)}
    \end{align*}
    where we used Agmon's inequality in 3D:
    $$
    \|u\|_{L^\infty}^2 \lesssim \|u\|_{L^2}^{1/2}\|u\|_{2}^{3/2}
    $$
    in the second inequality. We also invoked \eqref{est:regest1} with index $0$ to estimate $\|u\|_{L^2}$, \eqref{est:regest1} with index $k-1$ to bound $\|\srho{S-1}\|_{1}$, and \eqref{est:lem3uh2} to control $\|u\|_2$.

    Turning to the second inequality, since $\p_t^r u = 0$ on $\p\Omega$, then we can invoke Poincar\'e inequality to obtain:
    \begin{align*}
    \int_t^{T_*}\|R_{11}^r\|_1^2d\tau &\lesssim \int_t^{T_*}\|\nabla R_{11}^r\|_{L^2}^2d\tau\\
    &\lesssim \int_t^{T_*}\left( \|\nabla\su{r}\cdot\nabla\srho{S-1-r}\|_{L^2}^2 + \|\su{r}\cdot\nabla^2\srho{S-1-r}\|_{L^2}^2 \right) d\tau\\
    &=: R_{111}^r + R_{112}^r.
    \end{align*}
    If $1 \le r \le S-1$, using H\"older inequality and Gagliardo-Nirenberg-Sobolev inequalities, we can estimate $R_{111}^r$ by
    \begin{align*}
        R_{111}^r &\le \int_t^{T_*}\|\nabla\su{r}\|_{L^3}^2\|\nabla\srho{S-1-r}\|_{L^6}^2d\tau \lesssim \int_t^{T_*}\|\nabla\su{r}\|_{L^2}\|\nabla^2\su{r}\|_{L^2}\|\nabla\srho{S-1-r}\|_1^2d\tau\\
        &\le C(\rho_0, u_0, g,k)t^{-r}t^{-(k-2r)}\int_t^{T_*}\|\nabla^2\su{r}\|_{L^2}\|\nabla\srho{S-1-r}\|_1d\tau\\
        &\le C(\rho_0, u_0, g,k)t^{-r}t^{-\frac{k-2r}{2}}t^{-r}t^{-\frac{k-2r-1}{2}} = C(\rho_0, u_0, g,k)t^{-(k-\frac{1}{2})}.
    \end{align*}
    If $r = 0$, then we apply H\"older inequality and a Gagliardo-Nirenberg-Sobolev inequality to estimate that
    \begin{align*}
    R_{111}^0 &\le \int_t^{T_*} \|\nabla u\|_{L^3}^2 \| \nabla\srho{S-1}\|_{L^6}^2 d\tau \le \int_t^{T_*} \|u\|_1\|u\|_2\|\srho{S-1}\|_2^2 d\tau\\
    &\le C(\rho_0,u_0,g,k)t^{-1/2}\int_t^{T_*}\|\srho{S-1}\|_2^2 d\tau\\
    &\le C(\rho_0,u_0,g,k)t^{-1/2}t^{-(k-1)} \le C(\rho_0,u_0,g,k)t^{-(k-1/2)},
    \end{align*}
    where we used the bound \eqref{est:lem3uh2} and \eqref{est:regest2} with index $0$ and $k-1$ above. \\
    Now we discuss the bound for $R_{112}^r$. For $1 \le r \le S-1$, we have
    \begin{align*}
    R_{112}^r &\le \int_t^{T_*}\|\su{r}\|_{L^3}^2\|\nabla^2\srho{S-1-r}\|_{L^6}^2 d\tau\\
    &\le \int_t^{T_*}\|\su{r}\|_{L^2}\|\su{r}\|_1\|\srho{S-1-r}\|_3^2 d\tau\\
    &\le C(\rho_0, u_0, g, k)t^{-\frac{2r-1}{2}}t^{-r}\int_t^{T_*}\|\srho{S-1-r}\|_3^2 d\tau\\
    &\le C(\rho_0, u_0, g, k)t^{-(k-1/2)},
    \end{align*}
    where we used \eqref{est:regest1} with indices $2r-1$ and $2r$ in the third inequality, and \eqref{est:regest2} with index $k - 2r$ in the last inequality. If $r = 0$, then we take advantage of Agmon's inequality in 3D again to obtain:
    \begin{align*}
    R_{112}^0 &\le \int_t^{T_*}\|u\|_{L^\infty}^2\|\nabla^2\srho{S-1}\|_{L^2}^2 d\tau\\
    &\le \int_t^{T_*}\|u\|_{L^2}^{1/2}\|u\|_{2}^{3/2}\|\srho{S-1}\|_2^2 d\tau\\
    &\le C(\rho_0,u_0,g,k)t^{-3/4}\int_t^{T_*}\|\srho{S-1}\|_2^2 d\tau\\
    &\le C(\rho_0,u_0,g,k)t^{-3/4}t^{-(k-1)}\\
    &= C(\rho_0,u_0,g,k)t^{-(k-1/4)}.
    \end{align*}
    Therefore, we arrive at the bound:
    $$
    \int_t^{T_*}\|R_{11}^r\|_1^2d\tau \le C(\rho_0,u_0,g,k)t^{-(k-1/4)}.
    $$
    Proceeding in a similar fashion, we can acquire similar bounds for the $R_{12}^r$ and $R_{13}^r$. The proof of the lemma is thus complete after we sum up the estimates above.
    \end{proof}
    \end{lem}
    By Step 1, we know that for any $t \in (0,T_*]$,
    $$
    t^k\left(\|\srho{S}(t)\|_{L^2}^2 + \|\su{S}(t)\|_{L^2}^2\right) \le C(\rho_0, u_0, g,k),
    $$
    $$
    t^k\int_t^{T_*}\left(\|\srho{S}(\tau)\|_1^2 + \|\su{S}(\tau)\|_1^2\right)d\tau \le C(\rho_0, u_0, g,k).
    $$
    Combining Lemma \ref{lem3} with equations \eqref{srhoeqn2}, \eqref{sueqn2}, and using elliptic estimates, we conclude that for $t \in (0,T_*]$
    $$
    \|\srho{S-1}(t)\|_2^2 + \|\su{S-1}(t)\|_2^2 \le C(\rho_0, u_0, g,k)t^{-k},
    $$
    $$
    \int_t^{T_*}\left(\|\srho{S-1}(\tau)\|_3^2 + \|\su{S-1}(\tau)\|_3^2\right)d\tau \le C(\rho_0, u_0, g,k)t^{-k},
    $$
    which finishes the case when $l = S-1$. The rest will follow from an induction in $l$, and we omit the details here. Hence, we have concluded the case where $k$ is odd.

    \item \textbf{$k$ is even}. Since we have proved the $k=0$ case, we may write $k = 2S,$ $S \ge 1$, and define
    \[ \tilde{E}_s(t) = \|\nabla\srho{s}\|_{L^2}^2 + \|\nabla\su{s}\|_{L^2}^2 \]
     for $0 \le s \le S$. Notice that $\tilde{E}_s(t) \sim \|\srho{s}\|_1^2 + \|\su{s}\|_1^2$ in view of the Poincar\'e inequality.
     The scheme of the proof in this case is the same double induction argument (in forward $k$ and for each $k$ backwards in $l$), and we will follow the same outline as in the odd case.
     Considering \eqref{seqn} for $s = 1,\hdots, S$, we test \eqref{srhoeqn}, \eqref{sueqn} with $s= S$ by $-\Delta\srho{S},\calA\su{S}$ respectively, which yields:
    \begin{align*}
        \frac12\frac{d}{dt}\|\nabla\srho{S}\|_{L^2}^2 + \|\Delta \srho{S}\|_{L^2}^2 &= \sum_{r=0}^S{S\choose r}(\tilde{I}_r + \tilde{J}_r + \tilde{K}_r),\\
        \frac12\frac{d}{dt}\|\nabla \su{S}\|_{L^2}^2 + \|\calA\su{S}\|_{L^2}^2 &= g\int_\Omega \calA\su{S}\bP(\srho{S}e_z) \le \frac{g}{2}\tilde{E}_S,
    \end{align*}
    where for $r = 0,\hdots, S,$
    $$
    \tilde{I}_r = \int_\Omega \Delta\srho{S} (\su{r}\cdot \nabla)\srho{S-r},\; \tilde{J}_r = \int_\Omega \Delta\srho{S}\srho{r} \srho{S-r},
    $$
    $$
    \tilde{K}_r = \int_\Omega \Delta\srho{S}\nabla\srho{S-r}\cdot \nabla(-\Delta)^{-1}\srho{r}.
    $$

        To estimate $\tilde{I}_r$, we first observe that
        \begin{align*}
            \tilde{I}_r &\le \|\Delta \srho{S}\|_{L^2}\|\su{r}\|_{L^6}\|\nabla\srho{S-r}\|_{L^3} \le \|\Delta \srho{S}\|_{L^2}\|\su{r}\|_1\|\nabla\srho{S-r}\|_{L^2}^{1/2}\|\nabla^2\srho{S-r}\|_{L^2}^{1/2}.
        \end{align*}
        Hence if $r \ne 0$, we may estimate $\tilde{I}_r$ as follows: for any $\epsilon > 0$,
        $$
        \tilde{I}_r \le \epsilon\|\Delta \srho{S}\|_{L^2}^2 + C(\epsilon)\|\srho{S-r}\|_1\|\srho{S-r}\|_2\|\su{r}\|_1^2.
        $$
        If $r = 0,$ we estimate
        \begin{align*}
            \tilde{I}_0 &= \int_\Omega \Delta\srho{S}(u\cdot\nabla)\srho{S} \le \|\Delta\srho{S}\|_{L^2}\|u\|_{L^\infty}\|\nabla\srho{S}\|_{L^2}\le \epsilon\|\Delta \srho{S}\|_{L^2}^2 + C(\epsilon)\|u\|_2^2\|\nabla\srho{S}\|_{L^2}^2.
        \end{align*}

        To estimate $\tilde{J}_r$, we have:
        $$
        \tilde{J}_r \le \epsilon\|\Delta \srho{S}\|_{L^2}^2 + C(\epsilon)\|\srho{r}\|_1^2\|\srho{S-r}\|_1^2,
        $$
        where $\epsilon > 0$. Finally, to estimate of $\tilde{K}_r$, we evoke elliptic estimate to obtain
        \begin{align*}
        \tilde{K}_r &\le \|\Delta \srho{S}\|_{L^2}\|\srho{S-r}\|_1\|\nabla(-\Delta)^{-1}\srho{r}\|_{L^\infty} \lesssim \|\Delta \srho{S}\|_{L^2}\|\srho{S-r}\|_1\|\nabla(-\Delta)^{-1}\srho{r}\|_2\\
        &\le \epsilon\|\Delta \srho{S}\|_{L^2}^2 + C(\epsilon)\|\srho{r}\|_1^2\|\srho{S-r}\|_1^2,
        \end{align*}
        for any $\epsilon > 0$. Combining the estimates above yields the following energy inequality: for $\tau \in (0,T_*)$,
    \begin{align}
        \label{Senergyeven}
        \frac{d\tilde{E}_S}{d\tau} + \|\srho{S}(\tau,\cdot)\|_2^2 + &\|\su{S}(\tau,\cdot)\|_2^2 \le C(k)\bigg[ \left(g+\|u\|_2^2 + \|\rho\|_2^2 \right) \tilde{E}_S(\tau)\notag\\
        &+\sum_{r=1}^{S-1}\big(\|\srho{S-r}\|_1\|\srho{S-r}\|_2\|\su{r}\|_1^2 + \|\srho{S-r}\|_1^2\|\srho{r}\|_1^2\big)\bigg]\notag \\
        &\le C(k)\left(\tilde{F}(\tau)\tilde{E}_S(\tau) + \sum_{r=1}^{S-1}\tilde{G}_r(\tau)\right),
    \end{align}
    where $\tilde{F} \in L^1(0,T_*)$ due to the induction hypothesis at $k = 0$. Now, we would like to follow the same plan as that in the odd case. This motivates us to prove lemmas similar to Lemma \ref{lem1}, \ref{lem2}, and \ref{lem3} adapted to the even case. First, we show the following lemma that parallels Lemma \ref{lem1}:
    \begin{lem}
    \label{lem1even}
    There exists $\tau_0 \in [t/2,t]$ such that $\tilde{E}_S(\tau_0) \le C(\rho_0,u_0,g,k)t^{-k}$.
    \end{lem}
    \begin{proof}
We consider \eqref{seqn} with $s = S-1$. In view of \eqref{sueqn}, we have
\begin{align*}
\|\su{S}\|_1^2 &\lesssim \|\calA\su{S-1}\|_1^2 + g^2\|\srho{S-1}\|_1^2 \le \|\su{S-1}\|_3^2 + g^2\|\srho{S-1}\|_1^2,
\end{align*}
for any $\tau \in [t/2,t]$. Integrating over $[t/2,t]$ and using \eqref{est:regest2} with index $k-1$, we obtain
\begin{equation}
\label{lem1ueqeven}
\int_{t/2}^t \|\su{S}(\tau)\|_1^2 d\tau \le \int_{t/2}^{T_*} \|\su{S}(\tau)\|_1^2 d\tau\le C(\rho_0,u_0,g,k)t^{1-k}.
\end{equation}

To estimate $\|\nabla \srho{S}\|_{L^2}$, we apply $\nabla$ to both sides of \eqref{srhoeqn} with $s = S-1$, and then use H\"older's inequality:
    \begin{align}
    \label{srhointereven}
        \|\nabla\srho{S}\|_{L^2}^2 &\lesssim \|\srho{S-1}\|_{3}^2 + \sum_{r=0}^{S-1}C(k)\bigg(\|\nabla(\su{r}\cdot\nabla \srho{S-r-1})\|_{L^2}^2\notag\\
        &+ \|\nabla(\srho{S-r-1}\srho{r})\|_{L^2}^2 + \|\nabla(\nabla\srho{S-r-1}\cdot\nabla(-\Delta)^{-1}(\srho{r}))\|_{L^2}^2 \bigg).
    \end{align}
    To save space, we only consider the most singular term, namely $\|\nabla(\su{r}\cdot\nabla \srho{S-r-1})\|_{L^2}^2$, and show that
    \begin{equation}
    \label{est:lem1evenaux1}
    \int_{t/2}^t \|\nabla(\su{r}\cdot\nabla \srho{S-r-1})\|_{L^2}^2 d\tau \le C(\rho_0, u_0, g, k)t^{1-k}.
    \end{equation}
    The estimates on the rest of the terms follow from a similar argument. To show \eqref{est:lem1evenaux1}, we first compute that
    $$
    \nabla(\su{r}\cdot\nabla \srho{S-r-1}) = \nabla\su{r}\cdot\nabla \srho{S-r-1} + \su{r}\cdot\nabla^2 \srho{S-r-1}.
    $$
    The first term can be estimated by
    \begin{align*}
    \|\nabla\su{r}\cdot\nabla \srho{S-r-1}\|_{L^2}^2 &\le \|\nabla\su{r}\|_{L^4}^2\|\nabla \srho{S-r-1}\|_{L^4}^2\\
    &\lesssim \|\nabla\su{r}\|_1^2\|\nabla \srho{S-r-1}\|_1^2\\
    &\le \|\su{r}\|_2^2\|\srho{S-r-1}\|_2^2.
    \end{align*}
    Similarly, we may estimate the second term above by
    \begin{align*}
    \|\su{r}\cdot\nabla^2 \srho{S-r-1}\|_{L^2}^2 &\lesssim \|\su{r}\|_1^2\|\srho{S-r-1}\|_3^2
    \end{align*}
    Thus for $r = 0,\hdots, S-1$,
    \begin{align*}
    \int_{t/2}^t \|\nabla(\su{r}\cdot\nabla \srho{S-r-1})\|_{L^2}^2 d\tau &\lesssim \int_{t/2}^t \left( \|\su{r}\|_2^2\|\srho{S-r-1}\|_2^2 + \|\su{r}\|_1^2\|\srho{S-r-1}\|_3^2 \right) d\tau\\
    &\le C(\rho_0,u_0,g,k)(t^{-(2r+1)}t^{-(k-2r-2)} + t^{-2r}t^{-(k-2r-1)})\\
    &\le C(\rho_0,u_0,g,k)t^{-(k-1)}
    \end{align*}
    where we applied \eqref{est:regest1} with index $2r+1$ to $\|\su{r}\|_2$, \eqref{est:regest2} with index $k-2r-2$ to $\|\srho{S-1-r}\|_2$, \eqref{est:regest1} with index $2r$ to $\|\su{r}\|_1$, and \eqref{est:regest2} with index $k-2r-1$ to $\|\srho{S-1-r}\|_3$. In a similar fashion, we can also obtain the following bound:
    $$
    \int_{t/2}^t \left[\|\nabla(\srho{S-r-1}\srho{r})\|_{L^2}^2 + \|\nabla(\nabla\srho{S-r-1}\cdot\nabla(-\Delta)^{-1}(\srho{r}))\|_{L^2}^2\right]d\tau \le C(\rho_0,u_0,g,k)t^{1-k}.
    $$
    Collecting the estimates above and combining with \eqref{srhointereven}, we have
    \begin{equation}
    \label{lem1rhoeqeven}
    \int_{t/2}^t \|\nabla\srho{S}(\tau)\|_{L^2}^2d\tau \le C(\rho_0,u_0,g,k)t^{1-k}.
    \end{equation}
    Combining \eqref{lem1ueqeven} and \eqref{lem1rhoeqeven}, we have
    $$
    \int_{t/2}^t \tilde{E}_S(\tau) d\tau \le C(\rho_0,u_0,g,k)t^{1-k}.
    $$
    By mean value theorem, we can find a $\tau_0 \in (t/2,t)$ such that
    $$
    \tilde{E}_S(\tau_0)\le C(\rho_0,u_0,g,k)t^{-k},
    $$
    and this concludes the proof.
\end{proof}

Then we show a counterpart to Lemma \ref{lem2}.
    \begin{lem}
    \label{lem2even}
    Let $\tau_0$ be chosen as in Lemma \ref{lem1even}. Then for any $r = 1,\hdots, S-1$, we have
    $$
    \int_{\tau_0}^{T_*}\tilde{G}_r(\tau)d\tau \le C(\rho_0, u_0, g, k)t^{-(k-\frac12)}.
    $$
    \end{lem}
    \begin{proof}
    Observe that for $r = 1,\hdots, S-1$,
    $$
    \tilde{G}_r = \|\srho{S-r}\|_1\|\srho{S-r}\|_2\|\su{r}\|_1^2 + \|\srho{S-r}\|_1^2\|\srho{r}\|_1^2 =: \tilde{G}_r^1 + \tilde{G}_r^2.
    $$
   % Note that $\tilde{G}_r^2 = G_r^2$, then we estimate $\tilde{G}_r^2$ similar to $G_r^2$ as
   To estimate $\tilde{G}_r^2$, apply \eqref{est:regest1} with index $2r$ to $\|\srho{r}\|_1^2$ and \eqref{est:regest2} with index $k-2r-1$
   to $\|\srho{S-r}\|_1^2:$
    \begin{align*}
    \int_{\tau_0}^{T_*}\tilde{G}_r^2(\tau)d\tau &\le C(\rho_0, u_0, g,k)\tau_0^{-2r}\int_{\tau_0}^{T_*}\|\srho{S-r}\|_1^2 d\tau\\
    &\le  C(\rho_0, u_0, g,k)\tau_0^{-2r}\tau_0^{-(k-2r-1)}\\
    &\le C(\rho_0, u_0, g,k)t^{-(k-1)}.
    \end{align*}
    To treat the term $\tilde{G}_r^1$, we use the induction hypothesis to obtain that
    \begin{align*}
        \int_{\tau_0}^{T_*}\tilde{G}_r^1(\tau)d\tau &= \int_{\tau_0}^{T_*}\|\srho{S-r}\|_1\|\su{r}\|_1\|\srho{S-r}\|_2\|\su{r}\|_1d\tau\\
        &\le C(\rho_0, u_0, g,k)\tau_0^{-\frac{k-2r}{2}}\tau_0^{-r}\int_{\tau_0}^{T_*}\|\srho{S-r}\|_2\|\su{r}\|_1d\tau\\
        &\le C(\rho_0, u_0, g,k)\tau_0^{-\frac{k-2r}{2}}\tau_0^{-r}\tau_0^{-\frac{k-2r}{2}}\tau_0^{-\frac{2r-1}{2}}\\
        &\le C(\rho_0, u_0, g,k)t^{-(k-\frac12)}
    \end{align*}
    Summing up the two estimates above completes the proof of the lemma.
 %   We have thus showed Lemma \ref{lem2even} by summing the two estimates above.
    \end{proof}
Finally, we show a result parallel to Lemma \ref{lem3}.

    \begin{lem}
    \label{lem3even}
    For any $t \in (0,T_*]$,
    $$
    t^{k-\frac{1}{4}}(\|R_1(t)\|_1^2 + \|R_2(t)\|_1^2) \le C(\rho_0, u_0, g,k),
    $$
    $$
    t^{k-\frac{1}{4}}\int_{t}^{T_*}\left(\|R_1(\tau)\|_2^2 + \|R_2(\tau)\|_2^2\right)d\tau \le C(\rho_0, u_0, g,k),
    $$
    where $R_1, R_2$ are defined as in \eqref{seqn2}.
    \end{lem}
    \begin{proof}
    First, we note that by applying \eqref{est:regest1} and \eqref{est:regest2} with index $k-2$, we have
    $$
    t^{k-2}\left(\|R_2(t)\|_1^2 + \int_t^{T_*}\|R_2(\tau)\|_2^2 d\tau \right) \le C(\rho_0, u_0, g, k).
    $$
    Then it suffices for us to show suitable bounds for $R_1$. Similarly to the proof of Lemma \ref{lem3}, we need to control the following typical terms:
    $$
    R_{11}^r:=\su{r}\cdot\nabla\srho{S-1-r},\;R_{12}^r:=\srho{S-1-r}\srho{r},\;R_{13}^r:=\nabla\srho{S-1-r} \cdot \nabla (-\Delta)^{-1}(\srho{r}),
    $$
    For simplicity, we will only consider in detail the most singular term $R_{11}^r$, as the estimates for the remianing two terms will follow similarly.

    We first study $\|R_{11}^r\|_1^2$, and it suffices for us to consider the leading order contribution i.e. $\|\nabla R_{11}^r\|_{L^2}^2$. Recall from the proof of Lemma \ref{lem3} that
    $$
    \|\nabla R_{11}^r\|_{L^2}^2 \lesssim \|\nabla\su{r}\cdot\nabla\srho{S-1-r}\|_{L^2}^2 + \|\su{r}\cdot\nabla^2\srho{S-1-r}\|_{L^2}^2 =: R_{111}^r + R_{112}^r.
    $$
    To treat $R_{111}^r$, we see that for any $0 \le r \le S-1$ , an application of H\"older inequality, Sobolev embedding, and Gagliardo-Nirenberg Sobolev inequality yields:
    \begin{align*}
    R_{111}^r &\le \|\nabla \su{r}\|_{L^3}^2\|\nabla\srho{S-r-1}\|_{L^6}^2\\
    &\lesssim \|\nabla \su{r}\|_{L^2}\|\nabla \su{r}\|_1\|\nabla\srho{S-r-1}\|_{1}^2\\
    &\lesssim \|\su{r}\|_1\|\su{r}\|_2\|\srho{S-r-1}\|_{2}^2\\
    &\le C(\rho_0, u_0, g, k)t^{-r}t^{-\frac{2r+1}{2}}t^{-(k-2r-1)}\\
    &\le C(\rho_0, u_0, g, k)t^{-(k-1/2)},
    \end{align*}
    where we used \eqref{est:regest1} with indices $2r$, $2r+1$, $k-2r - 1$ respectively in the second to the last inequality above. To treat $R_{112}^r$, we first discuss the case when $1 \le r \le S-1$:
    \begin{align*}
    R_{112}^r &\le \|\su{r}\|_{L^3}^2\|\nabla^2\srho{S-r-1}\|_{L^6}^2\\
    &\le \|\su{r}\|_{L^2}\|\su{r}\|_1\|\srho{S-r-1}\|_3^2\\
    &\le C(\rho_0, u_0, g, k)t^{-\frac{2r-1}{2}}t^{-r}t^{-(k-2r)}\\
    &\le C(\rho_0, u_0, g, k)t^{-(k-1/2)},
    \end{align*}
    where we used \eqref{est:regest1} with index $2r-1$, $2r$, $k-2r$ respectively in the second to the last inequality above. In the case where $r = 0$, we instead estimate as follows using Agmon's inequality:
    \begin{align*}
    R_{112}^0 &\le \|u\|_{L^\infty}^2\|\nabla^2\srho{S-1}\|_{L^2}^2\\
    &\lesssim \|u\|_{L^2}^{1/2}\|u\|_2^{3/2}\|\srho{S-1}\|_2^2\\
    &\le C(\rho_0, u_0, g, k)t^{-\frac{3}{4}}t^{-(k-1)}\\
    &\le C(\rho_0, u_0, g, k)t^{-(k-1/4)},
    \end{align*}
    where we used \eqref{est:regest1} with index $0$, $1$, and $k-1$ in the third inequality. Combining the estimates above yields
    $$
    t^{k-1/4}\|R_{11}^r(t)\|_1^2 \le C(\rho_0, u_0, g, k).
    $$

    Now we shall study $\|R_{11}^r\|_2^2$. We still consider the leading order contribution, namely $\|\nabla^2 R_{11}^r\|_{L^2}^2$. A straightforward computation yields:
    \begin{align*}
    \|\nabla^2 R_{11}^r\|_{L^2}^2 &\lesssim \|\nabla^2\su{r} \cdot \nabla\srho{S-1-r}\|_{L^2}^2 + \|\nabla\su{r} \cdot \nabla^2\srho{S-1-r}\|_{L^2}^2 + \|\su{r} \cdot \nabla^3\srho{S-1-r}\|_{L^2}^2\\
    & =: \tilde{R}_{111}^r + \tilde{R}_{112}^r + \tilde{R}_{113}^r.
    \end{align*}
    To control $\tilde{R}_{111}^r$, we have for any $t \in (0,T_*]$:
    \begin{align*}
    \tilde{R}_{111}^r &\le \|\nabla^2\su{r}\|_{L^3}^2 \|\nabla\srho{S-1-r}\|_{L^6}^2\\
    &\lesssim \|\nabla^2\su{r}\|_{L^2}\|\nabla^2\su{r}\|_1\|\nabla\srho{S-1-r}\|_1^2\\
    &\lesssim \|\su{r}\|_{2}\|\su{r}\|_3\|\srho{S-1-r}\|_2^2\\
    &\le C(\rho_0, u_0, g, k)t^{-\frac{2r+1}{2}}t^{-\frac{k-2r-1}{2}}\|\su{r}\|_3\|\srho{S-1-r}\|_2\\
    &= C(\rho_0, u_0, g, k)t^{-\frac{k}{2}}\|\su{r}\|_3\|\srho{S-1-r}\|_2
    \end{align*}
    where we used \eqref{est:regest1} with indices $2r+1$ and $k-2r-1$ above. Integrating in time, we obtain
    \begin{align*}
    \int_t^{T_*}\tilde{R}_{111}^r d\tau &\le C(\rho_0, u_0, g, k)t^{-\frac{k}{2}}\int_t^{T_*}\|\su{r}\|_3\|\srho{S-1-r}\|_2 d\tau\\
    &\le C(\rho_0, u_0, g, k)t^{-\frac{k}{2}}\left(\int_t^{T_*}\|\su{r}\|_3^2 d\tau\right)^{1/2}\left(\int_t^{T_*}\|\srho{S-1-r}\|_2^2 d\tau\right)^{1/2}\\
    &\le C(\rho_0, u_0, g, k)t^{-\frac{k}{2}}t^{-\frac{2r+1}{2}}t^{-\frac{k-2r-2}{2}}\\
    &\le C(\rho_0, u_0, g, k)t^{-(k-1/2)},
    \end{align*}
    where we used \eqref{est:regest2} with indices $2r+1$ and $k-2r-2$. A similar argument switching the estimates of $u$ and $\rho$ terms yields the same bound for $\tilde{R}_{112}^r$:
    $$
     \int_t^{T_*}\tilde{R}_{112}^r d\tau \le C(\rho_0, u_0, g, k)t^{-(k-1/2)}.
    $$
    To estimate $\tilde{R}_{113}^r$, we first note that for $1 \le r \le S-1$,
    \begin{align*}
    \tilde{R}_{113}^r &\le \|\su{r}\|_{L^3}^2\|\nabla^3\srho{S-r-1}\|_{L^6}^2\\
    &\le \|\su{r}\|_{L^2}\|\su{r}\|_1\|\srho{S-r-1}\|_4^2\\
    &\le C(\rho_0, u_0, g, k)t^{-\frac{2r-1}{2}}t^{-r}\|\srho{S-r-1}\|_4^2
    \end{align*}
    where we used \eqref{est:regest1} with index $2r-1$ and $2r$ respectively in the last inequality above. Integrating in time, we get:
    \begin{align*}
    \int_t^{T_*}\tilde{R}_{113}^r d\tau &\le C(\rho_0, u_0, g, k)t^{-\frac{2r-1}{2}}t^{-r}\int_t^{T_*} \|\srho{S-r-1}\|_4^2 d\tau\\
    &\le C(\rho_0, u_0, g, k)t^{-\frac{2r-1}{2}}t^{-r}t^{k-2r}\\
    &=  C(\rho_0, u_0, g, k)t^{-(k-1/2)},
    \end{align*}
    where we used \eqref{est:regest2} with index $k - 2r$ above. In the case where $r = 0$, we instead estimate as follows using Agmon's inequality:
    \begin{align*}
    \tilde{R}_{113}^0 &\le \|u\|_{L^\infty}^2\|\nabla^3\srho{S-1}\|_{L^2}^2\\
    &\lesssim \|u\|_{L^2}^{1/2}\|u\|_2^{3/2}\|\srho{S-1}\|_3^2\\
    &\le C(\rho_0, u_0, g, k)t^{-\frac{3}{4}}\|\srho{S-1}\|_3^2
    \end{align*}
    where we used \eqref{est:regest1} with indices $0$ and $1$. Integrating in time yields:
    \begin{align*}
    \int_t^{T_*}\tilde{R}_{113}^0 d\tau &\le C(\rho_0, u_0, g, k)t^{-\frac{3}{4}}\int_t^{T_*}\|\srho{S-1}\|_3^2 d\tau\\
    &\le C(\rho_0, u_0, g, k)t^{-\frac{3}{4}}t^{-(k-1)}\\
    &= C(\rho_0, u_0, g, k)t^{-(k - 1/4)},
    \end{align*}
    where we used \eqref{est:regest2} with index $k-1$ above. Collecting the estimates above yields
    $$
    t^{k-1/4}\int_t^{T_*}\|\nabla^2 R_{11}^r\|_{L^2}^2 \le C(\rho_0, u_0, g, k).
    $$
    The proof is therefore completed.
    \end{proof}

     From this point on, a similar argument to the odd case combining with the three lemmas above finishes the proof for the even case. We leave details for the interested reader.
\end{enumerate}
\end{proof}
Finally, by combining Corollary \ref{cor:lowreg}, Proposition \ref{prop:highreg}, and using Sobolev embeddings, we infer the existence of a regular solution to \eqref{eq:ksstokes}

\subsection{Uniqueness}\label{subsect:uniqueness}
In this section, we show the uniqueness of regular solutions to problem \eqref{eq:ksstokes}.

\begin{prop}
\label{prop:uniqueness}
Given initial data $\rho_0 \in H_0^1$, $u_0 \in V$, there exist a $T_* > 0$ depending only on $\rho_0$, and a unique regular solution to problem \eqref{eq:ksstokes} on $[0, T_*]$.
\begin{proof}
Assume $(\rho_i, u_i)$, $i = 1,2$, to be two regular solutions to problem \eqref{eq:ksstokes} with initial condition $\rho_0, u_0$. Write $r = \rho_1 - \rho_2$, $w = u_1 - u_2$. A straightforward computation yields the following equations satisfied by $\rho, u$:
$$
\begin{cases}
    \p_t r - \Delta r + u_1\cdot\nabla r + w\cdot\nabla\rho_2 + \divv(r\nabla(-\Delta)^{-1}\rho_1-\rho_2\nabla(-\Delta)^{-1}r) = 0,\\
    \p_t w + \calA w = g\bP(r e_z),
    \end{cases}
$$
with boundary conditions $r|_{\p\Omega} = 0$, $w|_{\p\Omega} = 0$ and zero initial condition. Testing the $r$-equation by $r$, we obtain
\begin{align*}
    \frac{1}{2}\frac{d}{dt}\|r\|_{L^2}^2 + \|\nabla r\|_{L^2}^2 &= -\int_\Omega r u_1 \cdot\nabla r - \int_\Omega r (w\cdot \nabla \rho_2) + \int_\Omega r\nabla r\cdot \nabla(-\Delta)^{-1}\rho_1\\
    &- \int_\Omega \rho_2\nabla r \cdot \nabla(-\Delta)^{-1} r= I_1 + I_2 + I_3 + I_4.
\end{align*}
Using incompressibility of $u_1$, we immediately have $I_1 = 0$ via integration by parts. Using H\"older inequality and Sobolev embedding, we can estimate $I_2$ by:
$$
I_2 \le \|r\|_{L^2}\|w\|_{L^6}\|\nabla\rho_2\|_{L^3} \lesssim \|r\|_{L^2}\|w\|_{1}\|\rho_2\|_{2} \le \epsilon\|w\|_1^2 + C(\epsilon)\|\rho_2\|_2^2\|r\|_{L^2}^2
$$
for any $\epsilon > 0$. Using elliptic estimates, Sobolev embedding, and Gagliardo-Nirenberg-Sobolev inequalities,
we may estimate $I_3$ by:
\begin{align*}
I_3 &\le \|\nabla r\|_{L^2}\|r\|_{L^3}\|\nabla(-\Delta)^{-1}\rho_1\|_{L^6} \lesssim \|\nabla r\|_{L^2}\|r\|_{L^2}^{1/2}\|\nabla r\|_{L^2}^{1/2}\|\rho_1\|_{L^2}\\
&\lesssim \|\rho_1\|_{L^2}\|\nabla r\|_{L^2}^{3/2}\|\rho\|_{L^2}^{1/2}
\le \epsilon\|\nabla r\|_{L^2}^2 + C(\epsilon)\|\rho_1\|_{L^2}^4\|r\|_{L^2}^2.
\end{align*}
Similarly, we can estimate $I_4$ by
$$
I_4 \lesssim \|\rho_2\|_{L^\infty}\|r\|_{L^2}\|\nabla r\|_{L^2}\lesssim \|\rho_2\|_2\|r\|_{L^2}\|\nabla r\|_{L^2} \le \epsilon\|\nabla r\|_{L^2}^2 + C(\epsilon)\|\rho_2\|_2^2 \|r\|_{L^2}^2.
$$
On the other hand, we test the $w$-equation by $w$:
$$
\frac{1}{2}\frac{d}{dt}\|w\|_{L^2}^2 + \|\nabla w\|_{L^2}^2 = g\int_\Omega w\cdot r e_z \le \frac{1}{2}\|w\|_{L^2}^2 + \frac{g^2}{2}\|r\|_{L^2}^2.
$$
Consider $E(t) := \|w\|_{L^2}^2 + \|r\|_{L^2}^2$. Collecting the estimates above and choosing $\epsilon > 0$ to be sufficiently small, we have the following inequality:
$$
\frac{dE}{dt} \le C(\|\rho_2\|_2^2 + \|\rho_1\|_{L^2}^4 + g^2)E(t) =: Cf(t)E(t).
$$
Note that as $(\rho_i, u_i)$ are regular solutions for $i = 1,2$, we particularly have $\rho_1 \in C([0,T_*]; V)$ and $\rho_2 \in L^2((0,T_*); H^2 \cap V)$. Hence $f \in L^1(0,T_*)$. Since $(r,w)$ assumes zero initial condition, we have $E(0) = 0$. Then an application of Gr\"onwall's inequality implies
$$
E(t) = 0,\; t \in [0,T_*],
$$
and uniqueness is proved.
\end{proof}
\end{prop}

\subsection{Regularity Criterion}
\label{sec:crit}
In this section, we aim to prove Theorem \ref{thm:L2criterion}. We first need the following fact on the monotonicity of $L^1$ norm of cell density $\rho$:
\begin{lem}
\label{lem:L1lem}
Assume $\Omega$ to be a smooth domain in either $\R^2$ or $\R^3$. Let $(\rho,u)$ be a smooth solution to problem \eqref{eq:ksstokes} on $[0,T]$. Suppose also that $\rho_0$ is nonnegative. Then for any $t \in [0,T]$, we have
$$
\frac{d}{dt}\|\rho(t)\|_{L^1} \le 0.
$$
\begin{proof}
First, we note that by parabolic maximum principle, we must have $\rho(t,x) \ge 0$ in $[0,T]\times \Omega$. Using \eqref{eq:ksstokes}, we compute that
\begin{align*}
    \frac{d}{dt}\|\rho(t,\cdot)\|_{L^1} &= \frac{d}{dt}\int_\Omega \rho(t,x)dx = \int_\Omega \left( -u\cdot\nabla\rho + \Delta\rho - \divv(\rho\nabla(-\Delta)^{-1}\rho)\right) dx\\
    &=\int_\Omega \divv(\nabla\rho - \rho\nabla(-\Delta)^{-1}\rho)dx = \int_{\p \Omega}\frac{\p \rho}{\p n} - \rho\frac{\p}{\p n}(-\Delta)^{-1}\rho dS\\
    &= \int_{\p \Omega}\frac{\p \rho}{\p n} dS,
\end{align*}
where $\frac{\p}{\p n}$ denotes the outward normal derivative and $dS$ denotes the surface unit.
We also used the incompressibility of $u$, divergence theorem, and the Dirichlet boundary condition in the derivation above. In view of parabolic maximum principle, we must have
$$
\frac{ \p\rho}{\p n}\big|_{\p \Omega} \le 0.
$$
Hence, we conclude that
$$
\frac{d}{dt}\|\rho(t,\cdot)\|_{L^1}\le 0,\; t\in[0,T].
$$
\end{proof}
\end{lem}

Now, we are ready to give a proof of the $L^2$ regularity criterion:
\begin{proof}[Proof of Theorem \ref{thm:L2criterion}]
Assume $(\rho, u)$ is a solution to \eqref{eq:ksstokes} with smooth data $(\rho_0, u_0)$. Let $T_0 > 0$ be its maximal lifespan.
\begin{enumerate}
    \item $d=2$. Suppose $T_0 < \infty$ and
$$
\lim_{t \nearrow T_0}\int_0^t\|\rho\|_{L^2}^2ds = M < \infty.
$$
First, we test the $u$-equation in \eqref{eq:ksstokes} by $\calA u$, which yields:
\begin{align*}
\frac{1}{2}\frac{d}{dt}\|\nabla u\|_{L^2}^2 + \|\calA u\|_{L^2}^2 &= g\int_\Omega \calA u\cdot \rho e_2 \le \frac{1}{2}\|\calA u\|_{L^2}^2 + \frac{g^2}{2}\|\rho\|_{L^2}^2,\; t \in [0,T_0).
\end{align*}
Rearranging the above inequality, using Gr\"onwall inequality, Theorem \ref{stokesest} and the assumption, we obtain that
\begin{equation}
\label{criteria1}
\sup_{t \in [0,T_0]}\|u\|_1^2 + \int_0^{T_0}\|u\|_{2}^2 ds \le \|u_0\|_1^2 + \frac{g^2M}{2} < \infty.
\end{equation}
Testing $\rho$-equation by $-\Delta \rho$, one obtains that
\begin{align*}
    \frac{1}{2}\frac{d}{dt}\|\nabla\rho\|_{L^2}^2 + \|\Delta \rho\|_{L^2}^2 &= \int_\Omega \Delta \rho u\cdot \nabla\rho - \int_\Omega \Delta \rho \rho^2 + \int_\Omega \Delta\rho \nabla\rho\cdot\nabla(-\Delta)^{-1}\rho\\
    &=: Q_1 + Q_2 + Q_3.
\end{align*}
Similarly to the estimate \eqref{H1rho}, we have for any $\epsilon > 0$
$$
Q_1 \le \|\Delta \rho\|_{L^2}\|\nabla\rho\|_{L^2}\|u\|_{L^\infty} \le \epsilon \|\Delta \rho\|_{L^2}^2 + C(\epsilon)\|\nabla\rho\|_{L^2}^2\|u\|_2^2,
$$
$$
Q_2 \le \epsilon \|\Delta \rho\|_{L^2}^2 + C(\epsilon)\|\rho\|_{L^4}^4 \le  \epsilon \|\Delta \rho\|_{L^2}^2 + C(\epsilon)\|\rho\|_{L^2}^2\|\nabla\rho\|_{L^2}^2.
$$

The term that we have to treat differently is $Q_3$. Using H\"older inequality, Sobolev embedding, and an $L^p$-based elliptic estimate, we have:
\begin{align*}
    Q_3 &\le \|\Delta \rho\|_{L^2}\|\nabla\rho\|_{L^3}\|\nabla(-\Delta)^{-1}\rho\|_{L^6}\lesssim \|\Delta \rho\|_{L^2}\|\nabla\rho\|_{L^3}\|\nabla(-\Delta)^{-1}\rho\|_{1,\frac{3}{2}}\\
    &\lesssim \|\Delta \rho\|_{L^2}\|\nabla\rho\|_{L^3}\|\rho\|_{L^{3/2}}\lesssim \|\Delta \rho\|_{L^2}\|\rho\|_{L^2}^{1/3}\|\nabla^2 \rho\|_{L^2}^{2/3}\|\rho\|_{L^1}^{2/3}\|\nabla \rho\|_{L^2}^{1/3}\\
    &\lesssim \|\Delta \rho\|_{L^2}^{5/3}\|\rho\|_{L^2}^{1/3}\|\rho\|_{L^1}^{2/3}\|\nabla \rho\|_{L^2}^{1/3}\le \epsilon\|\Delta \rho\|_{L^2}^2 + C(\epsilon)\|\rho\|_{L^2}^{2}\|\rho\|_{L^1}^{4}\|\nabla \rho\|_{L^2}^{2},
\end{align*}
where we used the Gagliardo-Nirenberg-Sobolev inequalities
$$
    \|f\|_{L^{3/2}} \le C\|f\|_{L^1}^{2/3}\|\nabla f\|_{L^2}^{1/3},\;
    \|\nabla f\|_{L^{3}} \le C\|f\|_{L^2}^{1/3}\|\nabla^2 f\|_{L^2}^{2/3},
$$
in the fourth inequality, and Young's inequality in the last step. By Lemma \ref{lem:L1lem}, we know that for $t \in [0,T_0)$, $\|\rho(t,\cdot)\|_{L^1} \le \|\rho_0\|_{L^1}$. Then we have
$$
Q_3 \le  \epsilon\|\Delta \rho\|_{L^2}^2 + C(\rho_0, \epsilon)\|\rho\|_{L^2}^{2}\|\nabla \rho\|_{L^2}^{2}.
$$

Choosing $\epsilon > 0$ sufficiently small and using the estimates of $L_i$ above, the $\rho$-estimate can be rearranged as:
\begin{equation}
\label{criteria2}
\frac{d}{dt}\|\nabla\rho\|_{L^2}^2 + \|\Delta \rho\|_{L^2}^2 \le C(\rho_0) (\|u\|_2^2 + \|\rho\|_{L^2}^2)\|\nabla\rho\|_{L^2}^2.
\end{equation}
Using Gr\"onwall inequality, we have:
\begin{align*}
\sup_{0\le t \le T_0}\|\nabla\rho(t,\cdot)\|_{L^2}^2 + \int_0^{T_0}\|\rho\|_2^2ds &\lesssim \|\nabla\rho_0\|_{L^2}^2\exp\left(C(\rho_0)\int_0^{T_0}(\|u\|_2^2 + \|\rho\|_{L^2}^2) ds\right)\\
& \le C(\rho_0,u_0, M, g, T_0),
\end{align*}
where we used the assumption, \eqref{criteria1}, and elliptic estimate. But this implies that one can extend the solution $(\rho, u)$ beyond the supposed lifespan $T_0$ by Theorem \ref{thm:lwp}. This yields a contradiction.

\item $d=3$. Suppose $T_0 < \infty$ and
$$
\lim_{t \nearrow T_0}\int_0^t\|\rho\|_{L^2}^4ds = M < \infty.
$$
Testing the $u$-equation in \eqref{eq:ksstokes} by $Au$ and deploying estimates similar to the $d=2$ case, we have
$$
\sup_{t \in [0,T_0]}\|\nabla u\|_{L^2}^2 + \int_0^{T_0}\|u\|_{2}^2 ds \le \|u_0\|_1^2 + \frac{g^2\sqrt{MT_0}}{2} < \infty.
$$
A derivation identical to \eqref{H1rho} yields:
$$
    \frac{d}{dt}\|\nabla \rho\|_{L^2}^2 + \|\Delta \rho\|_{L^2}^2 \lesssim \left(\|\rho\|_{L^2}^4 + \|u\|_{2}^2\right)\|\nabla \rho\|_{L^2}^2.
$$
Applying Gr\"onwall inequality and combining the two estimates above, we have for $t \in [0,T_0]$ that
$$
\|\nabla\rho(t,\cdot)\|_{L^2}^2 + \int_0^{T_0}\|\rho\|_2^2ds \lesssim \|\rho_0\|_1^2\exp\left(C(\rho_0)\int_0^{T_0}(\|\rho\|_{L^2}^4 + \|u\|_{2}^2)ds\right) \le C(\rho_0,u_0,M,g,T_0).
$$
And this contradicts the assumption that $T_0$ is the maximal lifespan in view of Theorem \ref{thm:lwp}.
\end{enumerate}
The proof is thus completed.
\end{proof}

\section{Proof of the Main Theorem: Suppression of Chemotactic Blowup}
In this section, our goal is to conclude Theorem \ref{thm:main} that \eqref{eq:ksstokes} is globally regular in the regime of sufficiently large $g$. In particular, we will see that the coupling of the Keller-Segel equation to the Stokes flow with sufficiently robust buoyancy term is regularizing, in the sense that the solution $\rho(t,x)$ approaches zero exponentially fast as $g$ is sufficiently large. For the rest of the section, $\Omega$ denotes any smooth, bounded domain in either 2D or 3D.

\subsection{Velocity Control}
In this subsection, we remark on two controls on the velocity field $u$ in \eqref{eq:ksstokes} that will be instrumental in our main proof. The first lemma is in fact a standard $H^1_{t,x}$ control of $u$, which is hidden in our proof of energy estimate in Proposition \ref{prop:highreg}. We give a brief derivation here for clarity.
\begin{lem}
\label{lem:velest}
Let $(\rho, u)$ be a regular solution to problem \eqref{eq:ksstokes} with initial data $\rho_0 \in H^1_0$, $u_0 \in V$. We have
\begin{equation}
    \label{est:mainuest}
    \|u\|_{H^1([0,T_*]\times \Omega)}^2 \le C(\rho_0,u_0)(g^2 + 1).
\end{equation}

\begin{proof}
In view of the estimate \eqref{est:L2u} in Proposition \ref{prop:l2bd}, it suffices to show that
\begin{equation}
    \label{aux:dtu}
    \int_0^{T_*}\|\p_t u(t)\|_{L^2}^2 dt \le C(\rho_0,u_0)(g^2 + 1).
\end{equation}
Testing the $u$ equation in \eqref{eq:ksstokes} by $\p_tu$, we have
\begin{align*}
    \|\p_tu\|_{L^2}^2 + \frac{1}{2}\frac{d}{dt}\|\nabla u\|_{L^2}^2 &= g\int_\Omega \p_tu\cdot (\rho e_z)dx \le \frac{1}{2}\|\p_tu\|_{L^2}^2 + \frac{g^2}{2}\|\rho\|_{L^2}^2,
\end{align*}
where we used incompressiblity of $u$ and Cauchy-Schwarz inequality above. Rearranging, integrating in time, and using \eqref{est:L2rho} we obtain
\begin{align*}
    \int_0^t\|\p_tu(s)\|_{L^2}^2ds + \|\nabla u(t)\|_{L^2}^2 &\le g^2\int_0^t\|\rho(s)\|_{L^2}^2ds + \|\nabla u_0\|_{L^2}^2\\
    &\le g^2(2T_*\|\rho_0\|_{L^2}^2) + \|u_0\|_1^2\\
    &\le C(\rho_0,u_0)(g^2 + 1).
\end{align*}
By taking supremum of $t$ over $[0,T_*]$, we have arrive at the estimate \eqref{aux:dtu}.
\end{proof}
\end{lem}

The following lemma yields a key additional control over the velocity field by genuinely exploiting the buoyancy forcing structure of the fluid equation in \eqref{eq:ksstokes}:
\begin{lem}
\label{lem:key1}
Let $(\rho, u)$ be a regular solution to problem \eqref{eq:ksstokes} with initial data $\rho_0 \in H^1_0$, $u_0 \in V$. Then
\begin{equation}\label{aux829a}
\int_0^{T_*}\|u(t)\|_{L^2}^2 dt \le C(\Omega, \rho_0,u_0)(g+1).
\end{equation}
\begin{rmk}
    Note that a straightforward $L^2$ estimate of $u$ only yields a bound $\int_0^{T_*}\|u(t)\|_{L^2}^2 dt \lesssim g^2$. What we display in the lemma is that the structure of buoyancy forcing ``gains a $g^{-1}$''.
\end{rmk}
\begin{proof}
Without loss of generality, assume that $\Omega$ contains the origin.
Denote $L := \diam(\Omega)>0$. Multiplying the $\rho$-equation of \eqref{eq:ksstokes} by $z - L$ (recall that $z = x_d$ when $\Omega \subset \R^d,$ $d=2,3$) and integrating over $\Omega$, we have
\begin{align*}
    \frac{d}{dt}\int_\Omega (z-L)\rho dx+ \int_\Omega (z-L)(u\cdot \nabla\rho) dx- \int_\Omega (z-L)\Delta \rho dx+ \int_\Omega (z-L)\divv(\rho\nabla(-\Delta)^{-1}\rho) dx = 0.
\end{align*}
Moreover using the Dirichlet conditions $\rho|_{\p \Omega} = 0$ and $u|_{\p \Omega} = 0$, we note that via integration by parts:
$$
\int_\Omega (z-L)(u\cdot \nabla\rho) dx = -\int_\Omega \rho u_z dx+ \int_{\p \Omega}(z-L)\rho u_n dx= -\int_\Omega \rho u_z dx,
$$
$$
-\int_\Omega (z-L)\Delta \rho dx= \int_\Omega \p_z \rho dx- \int_{\p \Omega}(z-L)\frac{\p \rho}{\p n} d S,
$$
$$
\int_\Omega (z-L)\divv(\rho\nabla(-\Delta)^{-1}\rho) dx = -\int_\Omega \rho \p_z(-\Delta)^{-1}\rho dx,
$$
where $u_n$ denotes the normal component of $u$ along $\p \Omega$, and $dS$ denotes the surface measure induced on $\p\Omega$. Collecting the above computations, we have
\begin{align}
    \int_\Omega \rho u_z dx &= \frac{d}{dt}\int_\Omega (z-L)\rho dx+ \int_\Omega \p_z \rho dx- \int_{\p \Omega}(z-L)\frac{\p \rho}{\p n} dS -\int_\Omega \rho \p_z(-\Delta)^{-1}\rho dx. \label{lem:key1eqn1}
\end{align}
On the other hand, testing the $u$-equation of \eqref{eq:ksstokes} by $u$, we also have
\begin{equation}
    \frac{1}{2}\frac{d}{dt}\|u\|_{L^2}^2 + \|\nabla u\|_{L^2}^2 = g\int_\Omega \rho u_z.\label{lem:key1eqn2}
\end{equation}
From Lemma \ref{lem:L1lem}, we also know that $\p \rho/\p n \le 0$ on $\p \Omega$ in $[0,T_*]$. Hence, we have $\int_{\p \Omega}(z-L)\frac{\p \rho}{\p n} dS \ge 0$ by definition of $L$. Combining this fact with \eqref{lem:key1eqn1}, \eqref{lem:key1eqn2}, and integrating on $[0,T_*]$, we have
\begin{align*}
    \|u(t)\|_{L^2}^2-\|u_0\|_{L^2}^2 &\le 2g\bigg[\int_\Omega (z-L)(\rho(t,x) - \rho_0(x))\,dx + \int_0^t\int_\Omega \p_z \rho \,dx - \int_0^t\int_\Omega \rho \p_z(-\Delta)^{-1}\rho \,dx\bigg]\\
    &\le C(\Omega)g(\|\rho_0\|_{L^1} + \sqrt{T_*}\left(\int_0^{T_*}\|\nabla\rho\|_{L^2}^2 dt\right)^{1/2} + \int_0^{T_*}\|\rho\|_{L^2}^2 dt)\\
    &\le C(\Omega, \rho_0)g,
\end{align*}
where we used elliptic estimate in the second inequality, and \eqref{est:L2rho} in the final inequality. The proof is therefore completed after integrating in time again.
\end{proof}
\end{lem}

\subsection{A Key Theorem}
In this part, we prove a quantitative characterization of the regularizing effect of the Stokes-Boussinesq flow in \eqref{eq:ksstokes}. With a rigidity-type argument inspired by \cite{CNR}, we show that the flow with sufficiently large $g$ can suppress the $L^2$ energy of $\rho$ to be arbitrarily small within the time scale of local existence, as elucidated in the following theorem:
\begin{thm}
\label{key}
Let $\rho_0 \in H^1_0, u_0 \in V$ be initial conditions for the problem \eqref{eq:ksstokes}, and consider $(\rho, u)$ to be the regular solution. For arbitrary $\epsilon > 0$, there exists $g_* = g_*(\rho_0, u_0, \epsilon)$ such that for any $g \ge g_*$,
$$
\inf_{t \in [0,T_*]} \|\rho(t,\cdot)\|_{L^2} \le \epsilon.
$$
\end{thm}
\begin{proof}
Suppose for the sake of contradiction that there exists $\epsilon_0 > 0$ such that there is a sequence $\{(\rho_n, u_n, g_n)\}_n$ which are regular solutions to \eqref{eq:ksstokes} with $\rho = \rho_n, u=u_n, g = g_n$ and $g_n \to +\infty$ (corresponding to initial data $\rho_0, u_0$). Indeed, we may without loss of generality assume that the sequence $\{g_n\}_n$ is increasing by picking a subsequence if necessary. Also, for any $t \in [0,T_*]$, and for all $n$
\begin{equation}
\label{contra}
\|\rho_n(t,\cdot)\|_{L^2} > \epsilon_0.
\end{equation}
Note that indeed we can use the uniform choice of time $T_*$ here, since $T_*$ only depends on $\rho_0$. Moreover, we consider the normalized velocity $\bu_n = u_n/g_n$. We will divide the proof into the following steps:
\begin{itemize}
    \item \textbf{Step 1: Convergence properties of $(\rho_n, u_n)$.} From \eqref{est:mainuest}, we have $\|\bu_n\|_{H^1([0,T_*]\times \Omega)}\le C(\rho_0,u_0)$. Using weak compactness and the Sobolev compact embedding theorem, we obtain that there exists $\bu_\infty \in H^1([0,T_*]\times \Omega)$ such that
    $$
    \bu_n \rightharpoonup \bu_\infty\; \text{ in } H^1([0,T_*]\times \Omega),\,\,\,{\rm and}\,\,\,\bu_n \to \bu_\infty\; \text{ in } L^2([0,T_*]\times \Omega).
    $$
    In fact, observe that from the estimate \eqref{aux829a} of Lemma \ref{lem:key1} it follows that $\|\bu_n\|_{L^2([0,T_*]\times \Omega)} \rightarrow 0$ as $n \rightarrow \infty,$
    so $\bu_\infty =0.$
    In addition, from the energy estimate \eqref{est:L2rho}, we may pick a further subsequence, still indexed by $n$, such that there exists $\rho_\infty \in L^2(0,T_*; H^1_0(\Omega))$ and
    $$
    \rho_n \rightharpoonup \rho_\infty\; \text{ in } L^2(0,T_*; H^1_0(\Omega)).
    $$
    \item \textbf{Step 2: Derivation of the limiting fluid equation.}
    Since $(\rho_n,u_n)$ is a regular solution to \eqref{eq:ksstokes} with parameter $g_n$ on $[0,T_*]$, $u_n$ in particular solves the fluid equation in \eqref{eq:ksstokes} weakly.
    That is,
    $$
    -\int_0^{T_*}\int_\Omega (\p_t \phi)\bu_n dxdt + \int_0^{T_*}\int_\Omega (\calA\phi)\bu_n dxdt = \int_0^{T_*}\int_\Omega \rho_n(\phi\cdot e_z) dxdt,
    $$
    for any smooth vector field $\phi \in C_c^\infty([0,T_*] \times \Omega)$ with $\divv\phi = 0$.
    By the convergence properties of $\rho_n$, $u_n$ as shown in Step 1, and by Lemma \ref{lem:key1} we find that
    %$$
    %-\int_0^{T_*}\int_\Omega (\p_t \phi)\bu_\infty dxdt + \int_0^{T_*}\int_\Omega (\calA\phi)\bu_\infty dxdt = \int_0^{T_*}\int_\Omega \rho_\infty(\phi\cdot e_z) dxdt,
    %$$
    \begin{equation}\label{weakeqn}
    \rho_\infty e_z = \nabla p_\infty, \,\,\, (t,x) \in [0,T_*] \times \Omega
    \end{equation}
    holds in a weak sense.
   % Moreover, we observe the initial condition: $\bu_n(0,x) = u_0(x)/g_n$. That is, $\bu_\infty$ weakly solve the Stokes equation with unit Rayleigh number:
   % \begin{equation}
   % \label{weakeqn}
   % \begin{cases}
   % \p_t \bu_\infty - \Delta \bu_\infty = -\nabla p_\infty + \rho_\infty  e_z,&(t,x) \in [0,T_*] \times \Omega\\
   % \divv\bu_\infty = 0,&(t,x) \in [0,T_*] \times \Omega\\
   % \bu_\infty(0,x) = 0,\;\bu_\infty|_{\p \Omega} = 0.
   % \end{cases}
   %  \end{equation}

    \item \textbf{Step 3: Nontriviality of $\rho_\infty$.} By maximum principle, we know that $\rho_n$, and thus $\rho_\infty$, is nonnegative. We would also like to claim that $\rho_\infty \not\equiv 0$. To show this fact, we need the following proposition.
 \begin{prop}
\label{prop:l2tolinfty}
Let $\Omega \subset \R^d$, $d = 2,3$, be a smooth, bounded domain. Assume $(\rho, u)$ to be the regular solution of problem \eqref{eq:ksstokes} on $[0,T_*]$ with initial condition $\rho_0 \geq 0 \in H_0^1,$ $u_0 \in V$.
If there exists $M > 0$ such that $\sup_{0\le t \le T_*}\|\rho(t)\|_{L^2} \le M$, then we have
$$
\sup_{0\le t\le T_*}\|\rho(t)\|_{L^\infty} \le CM^{\frac{4}{4-d}}.
$$
Here $C$ is a constant that may only depend on $d$ and $\Omega$.
\end{prop}
A variant of this result has been proved in \cite{KX} (Proposition 9.1), in a two dimensional periodic setting.
The proof of Proposition \ref{prop:l2tolinfty} is similar and for the sake of completeness will be provided in the appendix.
%    record two auxiliary propositions below, whose proofs will be delayed to the appendix.

Next, we need the following lemma.
\begin{lem}
\label{prop:nonneg2}
Let $D \subset \R^d$, $d\in \N$, be a bounded domain, and let $\{f_n\}_n \subset L^2(D)$ be a sequence of nonnegative functions that weakly converges to a function $f \in L^2(D)$. Assume that there exist $M, \epsilon > 0$ such that $\|f_n\|_{L^2} > \epsilon$, $\|f_n\|_{L^\infty} \le M$ for all $n$. Then $f \not\equiv 0$.
\end{lem}
\begin{proof}
Suppose for the sake of contradiction that $f \equiv 0$. Consider the characteristic function $\phi = \chi_{D}$ Since $D$ is bounded, $\phi \in L^2(D)$. Then the weak convergence informs us that
$$
\lim_{n \to \infty}\int_D f_n = 0.
$$
As $f_n \ge 0$ for all $n$, this is equivalent to $\lim_{n \to \infty}\|f_n\|_{L^1} = 0$. Since $\|f_n\|_{L^\infty} \le M$, by interpolation we have
$$
\|f_n\|_{L^2}^2 \le \|f_n\|_{L^\infty}\|f_n\|_{L^1} \to 0
$$
as $n \to \infty$. But this contradicts with the assumption that $\|f_n\|_{L^2} > \epsilon$.
\end{proof}

Observe that from \eqref{est:L2rho}, we know that $\|\rho_n(t,\cdot)\|_{L^2} \leq 4 \|\rho_0\|_{L^2}$ for all $t \in [0,T_*]$ and all $n.$
Thus applying Proposition \ref{prop:l2tolinfty} to $\rho_n$ we get that $\|\rho_n(t, \cdot)\|_{L^\infty} \leq M$ for all $t \in [0,T_*],$ and all $n$, where
$M = C(d,\Omega)\|\rho_0\|_{L^2}^{\frac{4}{4-d}}.$ Then Lemma \ref{prop:nonneg2} implies that $\rho_\infty \not\equiv 0.$

%The claim that $\rho_\infty \not\equiv 0$ holds by first applying Proposition \ref{prop:l2tolinfty} to $\rho_n$, and then by applying Proposition \ref{prop:nonneg2} with $f_n =\rho_n$ and $f = \rho_\infty$.

\item \textbf{Step 4: Derivation of a contradiction.}
 Let us consider
    $$
    \psi_n(x) := \int_0^{T_*}\rho_n(t,x)dt,\; \psi_\infty(x) := \int_0^{T_*}\rho_\infty(t,x)dt.
    $$
    In particular, $\psi_\infty \not\equiv 0$ and $\psi_\infty \geq 0$ by Step 3. On one hand, picking arbitrary $\eta \in L^2(\Omega)$, we have
    \begin{align*}
    \bigg|\int_\Omega \eta(x)(\psi_n(x) - \psi_\infty(x))dx\bigg| &= \bigg|\int_0^{T_*}\int_\Omega \eta(x)(\rho_n(t,x) - \rho_\infty(t,x))dxdt\bigg|\\
    &= \bigg|\int_0^{T_*}\int_\Omega \eta(x)\chi_{[0,T_*]}(t)(\rho_n(t,x) - \rho_\infty(t,x))dxdt\bigg|,
    \end{align*}
    which converges to $0$ as $\rho_n \rightharpoonup \rho_\infty\; \text{ in } L^2([0,T_*]\times \Omega)$. This implies that $\psi_n \rightharpoonup \psi_\infty$ in $L^2( \Omega)$. On the other hand, we note that by Minkowski inequality and H\"older inequality,
    $$
    \|\nabla \psi_n\|_{L^2} \le \int_0^{T_*}\|\nabla\rho_n\|_{L^2}dt \le \sqrt{T_*}\|\nabla\rho_n\|_{L^2([0,T_*]\times \Omega)} \le C(\rho_0),
    $$
    where we used \eqref{est:L2rho} in the last step. Since $\rho_n|_{\p \Omega} = 0$, we know that $\psi_n \in H^1_0(\Omega)$ with a uniform $H^1$-norm bound from above. Hence by weak compactness and Sobolev compact embedding theorem, there exists a subsequence, still denoted by $\psi_n$, and $\tilde\psi_\infty \in H^1_0(\Omega)$ such that
    $$
    \psi_n \rightharpoonup \tilde{\psi}_\infty\; \text{ in } H^1_0(\Omega),\;\psi_n \to \tilde{\psi}_\infty\; \text{ in } L^2(\Omega).
    $$
    Indeed, we must have $\tilde{\psi}_\infty = \psi_\infty$ due to the uniqueness of weak limit, and hence $\psi_\infty \in H^1_0(\Omega).$

    %By Lemma \ref{lem:key1}, we know that $\lim_{n\to\infty}\int_0^{T_*}\|\bu_n(t,\cdot)\|_{L^2}^2dt = 0$. Since $\bu_\infty$ solves \eqref{weakeqn} weakly, we obtain that there exists $p_\infty \in L^2([0,T]\times \Omega)$ such that (in the sense of distribution)
    %$$
    %\nabla p_\infty = \rho_\infty  e_z.
    %$$
    But now, integrating \eqref{weakeqn} with respect to time, we have
    $$
    \nabla P = \psi_\infty  e_z,
    $$
    where $P(x) := \int_0^{T_*}p_\infty(t,x)dt$. But this implies that $\psi_\infty(x) = h(z)$, where $h$ is some single-variable function. Moreover, we know that $\psi_\infty \in H^1_0(\Omega)$. These two facts imply that $\psi_\infty \equiv 0$. However, this contradicts the fact that $\psi_\infty > 0$. This completes the proof of the theorem.
\end{itemize}

\end{proof}

\subsection{Proof of Global Well-Posedness with Large $g$}
Now we are ready to prove the main theorem. As we will see below, $\|\rho\|_{L^2}$ enjoys a Riccati-type differential inequality which preserves smallness. This structure combined with Theorem \ref{key} gives the boundedness of $\|\rho(t)\|_{L^2}^2$ globally in time (and actually smallness in large time). The proof is thus done after we invoke Theorem \ref{thm:L2criterion}.

\begin{proof}[Proof of Theorem \ref{thm:main}] %Let regular initial data $(\rho_0,u_0)$ be given and $d= 2$ or $3$.
Say the solution $(\rho, u)$ is regular up to a maximal time $T_0$, which may or may not be infinite. Suppose first that $T_0 < \infty$.
Similarly to the proof of Proposition \ref{prop:l2bd}, using the energy estimate of $\rho$, a Gagliardo-Nirenberg-Sobolev inequality, Young's inequality and Poincaré inequality, we have for $t \in (0,T_0)$ that
\begin{align}
  \frac{1}{2}\frac{d}{dt}\|\rho\|_{L^2}^2 &\le -\|\nabla\rho\|_{L^2}^2 + \frac{1}{2}\|\nabla\rho\|_{L^2}^2 + C\|\rho\|_{L^2}^{\frac{12-2d}{4-d}} \nonumber \\
  &\le -\frac{1}{2C_p}\|\rho\|_{L^2}^2 + C\|\rho\|_{L^2}^{\frac{12-2d}{4-d}} =: f_d(\|\rho\|_{L^2}^2), \label{aux829c}
\end{align}
where $C_p$ denotes the Poincaré constant that only depends on domain $\Omega$. Since $2 < \frac{12-2d}{4-d}$ when $d = 2,3$, we fix $\epsilon \in (0,1)$ sufficiently small that $f_d(\epsilon) < -\frac{1}{4C_p}\epsilon$. Note that such choice of $\epsilon$ only depends on domain $\Omega$. By Theorem \ref{key}, there exists $g_* = g_*(\rho_0,u_0)$ such that there exists $\tau \in [0,T_*]$ with
$
\|\rho(\tau)\|_{L^2}^2 \le \epsilon
$
for any $g \ge g_*$. Now we consider the problem \eqref{eq:ksstokes} starting from $t = \tau$. Then from the inequalities above, we note that $\frac{d}{dt}\|\rho(t,\cdot)\|_{L^2}^2|_{t = \tau} < 0;$ by \eqref{aux829c}
this inequality also holds for all $t \in [\tau,T_0].$
%Applying Gr\"onwall inequality, we can conclude that
%$
%\sup_{t \in [\tau, T_0]}\|\rho(t,\cdot)\|_{L^2}^2 \le \epsilon.
%$
Hence, there exists $M > 0$ depending on $\rho_0$ such that
$
\sup_{t \in [0,T_0]}\|\rho(t,\cdot)\|_{L^2}^2 \le M,
$
which yields
$$
\int_{0}^{T_0}\|\rho(t,\cdot)\|_{L^2}^{\frac{4}{4-d}} dt \le M^{\frac{2}{4-d}}T_0 < \infty.
$$
By the regularity criterion, this contradicts the definition of $T_0.$ Therefore, we conclude that $T_0 = \infty$ and the solution is globally regular. To prove \eqref{eq:quenching}, we note from above that $\sup_{t \ge \tau}\|\rho(t,\cdot)\|_{L^2}^2 \le \epsilon$.
In fact, by our choice of $\epsilon$ and \eqref{aux829c}, we have
$$
\frac{d}{dt}\|\rho\|_{L^2}^2 \le -\frac{1}{4C_p}\|\rho\|_{L^2}^2,\; t \ge \tau.
$$
Using Gr\"onwall inequality, $\|\rho(\tau)\|_{L^2}^2 \le \epsilon < 1$, and $\tau \le T_* \le 1$, we have: for $t \ge \tau$,
\begin{align*}
\|\rho(t)\|_{L^2}^2 &\le \|\rho(\tau)\|_{L^2}^2e^{-\frac{1}{4C_p}(t-\tau)} \le e^{-\frac{1}{4C_p}t}e^{\frac{1}{4C_p}\tau}\\
&\le Ce^{-\frac{1}{4C_p}t},
\end{align*}
where $C = e^{1/4C_p}$ is a constant that only depends on domain $\Omega$. This yields \eqref{eq:quenching} after rearranging the inequality above.

%Therefore $\|\rho(t,\cdot)\|_{L^2}$ decays exponentially as $t \rightarrow \infty.$
%An application of Gr\"onwall inequality immediately yields \eqref{eq:quenching}.
\end{proof}

\appendix
\section{Appendix}
In the appendix, we will remark on one regularity estimate for Stokes operator $\calA$ that plays an essential role in our energy estimates. What follows will be a proof of Proposition \ref{prop:l2tolinfty} that appear in the proof of the main lemma.

The regularity result for Stokes operator stated below is standard; proofs can be found for example in \cite{CF}:
\begin{athm}
\label{stokesest}
Let $\Omega$ be a bounded $C^2$ domain. Then there exists a constant $C = C(\Omega)$ such that for all $u \in D(\calA) = H^2(\Omega) \cap H^1_0(\Omega)$,
$$
\|u\|_{2} \le C(\Omega)\|\calA u\|_{L^2}.
$$
Moreover, there exist constants $c_0, C_0$ only depending on domain $\Omega$ such that
$$
c_0\|\nabla u\|_{L^2} \le \|\calA^{1/2}u\|_{L^2} \le C_0\|\nabla u\|_{L^2}
$$
\end{athm}

We will now give a proof for Proposition \ref{prop:l2tolinfty}; its statement is reiterated below.
\begin{aprop}
Let $\Omega \subset \R^d$, $d = 2,3$, be a smooth, bounded domain. Assume $\rho, u$ to be the regular solution of problem \eqref{eq:ksstokes} on $[0,T]$ with initial condition $\rho_0 \in H_0^1$ and $\rho_0 \ge 0$. If there exists $M > 0$ such that $\sup_{0\le t \le T}\|\rho(t)\|_{L^2} \le M$, we then have
$$
\sup_{0\le t\le T}\|\rho(t)\|_{L^\infty} \le CM^{\frac{4}{4-d}},
$$
where $C$ is a constant only depending on domain $\Omega$.
\begin{proof}
In the proof, we shall suppress the variable $t$. Let $p \ge 1$ be an integer. We start with the following computation using \eqref{eq:ksstokes}:
\begin{align*}
    \frac{d}{dt}\|\rho\|_{L^{2p}}^{2p} = 2p\int_\Omega \rho^{2p-1}(-(u\cdot\nabla) \rho - \divv(\rho\nabla(-\Delta)^{-1}\rho) + \Delta\rho)\,dx = 2p(I+J+K).
\end{align*}
Using incompressibility of $u$, we can compute that
\begin{align*}
%    I &= -\int_\Omega (2p-1)\rho^{2p - 2}\p_j\rho u_j \cdot \rho dx= -(2p-1)\int_\Omega\rho^{2\rho - 1}u_j\p_j\rho dx = -(2p-1)I,
I = - \int_{\Omega}  \rho^{2p-1} (u\cdot\nabla) \rho \, dx = - \frac{1}{2p} \int_{\Omega} u_j \partial_j \rho^{2p} =0.
\end{align*}
%which implies $I = 0$ as $p \ge 1$.
Integrating by parts, we have
\begin{align*}
    J &= (2p-1)\int_\Omega \rho^{2p-1}\p_j\rho \p_j(-\Delta)^{-1}\rho dx= \frac{2p-1}{2p}\int_\Omega\p_j(\rho^{2p})\p_j(-\Delta)^{-1}\rho dx= \frac{2p-1}{2p}\int_\Omega \rho^{2p+1}dx
\end{align*}
Using chain rule, we also have
\begin{align*}
    K &= -(2p-1)\int_\Omega \rho^{2p-2}\p_j\rho\p_j\rho dx= -\frac{2p-1}{p^2}\int_\Omega |\nabla \rho^p|^2 dx.
\end{align*}
Collecting all computations above, we observe that
\begin{equation}
\label{Linftybd1}
\frac{d}{dt}\|\rho\|_{L^{2p}}^{2p} = (2p-1)\|\rho\|_{L^{2p+1}}^{2p+1} - \left(4-\frac{2}{p}\right)\|\nabla\rho^p\|_{L^2}^2.
\end{equation}
Now we shall estimate $\|\rho\|_{L^{2^n}}$ inductively on $n$. The base case $n=1$ is dealt with by our assumption. Assume for $t \in [0,T]$ we have the bound
$$
\|\rho\|_{L^{2^n}} \le B_n, B_n \ge 1
$$
for any $t \in [0,T]$. Define $f = \rho^{2^n}$, and apply $p = 2^n$ to \eqref{Linftybd1}, we obtain that
\begin{equation}
    \label{Linftydiff}
    \frac{d}{dt}\int_\Omega f^2 dx \le -2\|\nabla f\|_{L^2}^2 + 2^{n+1}\|f\|_{L^{2+2^{-n}}}^{2+2^{-n}}.
\end{equation}
Applying a Gagliardo-Nirenberg-Sobolev inequality (see \cite{AF}, for example), we can estimate using Young's inequality that
%\blue{I fixed the wrong exponent in A.3}
\begin{align}
\|f\|_{L^{2+2^{-n}}}^{2+2^{-n}} \lesssim \|\nabla f\|_{L^2}^{d2^{-n-1}}\|f\|_{L^2}^{2+2^{-n}-d2^{-n-1}} \le d2^{-n-2}\|\nabla f\|_{L^2}^2 + C\|f\|_{L^2}^{\frac{2+2^{-n} -d2^{-n-1}}{1-d 2^{-n-2}}}, \label{Linftybd2}\\
\|f\|_{L^2} \lesssim \|\nabla f\|_{L^2}^{\frac{d}{d+2}}\|f\|_{L^1}^{\frac{2}{d+2}}.\label{Linftybd3}
\end{align}
The constants in the above inequalities do not depend on $n.$
Plugging \eqref{Linftybd2}, \eqref{Linftybd3} to \eqref{Linftydiff}, we obtain
\begin{align}
    \frac{d}{dt}\int_\Omega f^2 dx &\le -2\|\nabla f\|_{L^2}^2 +\frac{d}{2} \|\nabla f\|_{L^2}^2 + C_2 2^{n+1}\|f\|_{L^2}^{\frac{2+2^{-n} -d2^{-n-1}}{1-d 2^{-n-2}}}\notag\\
    &\le -C_1\|f\|_{L^2}^{\frac{2d+4}{d}}\|f\|_{L^1}^{-\frac{4}{d}} + C_2 2^{n+1}\|f\|_{L^2}^{\frac{2+2^{-n} -d2^{-n-1}}{1-d 2^{-n-2}}}, \label{Linftydiff2}
\end{align}
where $C_1$, $C_2$ are constants only depending on $d$. Note that given $d = 2,3$, we have  $\frac{2+2^{-n} -d2^{-n-1}}{1-d 2^{-n-2}} < \frac{2d+4}{d}$ for $n \ge 1$. Moreover, observe that
$$
\|f\|_{L^1} \le B_n^{2^{n}} < \infty.
$$
Then for each $n \ge 1$, the right hand side of \eqref{Linftydiff2} becomes negative when $\|f\|_{L^2}$ is sufficiently large. In particular, one can compute that $\|\rho\|_{L^{2^{n+1}}}$ will never reach the value $B_{n+1}$, where
$B_{n+1}$ is defined by the following recursive relation:
$$
\log B_{n+1} = \frac{2^{n+2}-d}{2^{n+2}-2d}\log B_n + \frac{d}{2^{n}} \left[\log C + (n+1)\log 2\right],
$$
where $C$ is a constant independent of $n$. Note that we have
$$
\prod_{j=1}^n \frac{2^{n+2}-d}{2^{n+2}-2d} = \frac{4-d2^{-n}}{4-d} \to \frac{4}{4-d}
$$
as $n \to \infty$, where in the first equality we used the telescoping nature of the product. Then via an inductive argument, there exists some dimensional constant $C > 0$ such that for all $n \ge 1$,
$$
B_n \le CM^{\frac{4}{4-d}}.
$$
As $\Omega$ is bounded, we have
$$
\|\rho\|_{L^\infty} = \lim_{n \to \infty}\|\rho\|_{L^{2^n}} \le CM^{\frac{4}{4-d}},
$$
and the proof of the lemma is complete.
\end{proof}
\end{aprop}

%\begin{aprop}
%Let $D \subset \R^d$ be a bounded domain, and let $\{f_n\}_n \subset L^2(D)$ be a nonnegative sequence that weakly converges to $f \in L^2(D)$. Assume there exists $M, \epsilon > 0$ such that $\|f_n\|_{L^2} > \epsilon$, $\|f_n\|_{L^\infty} \le M$ for all $n$, then $f \not\equiv 0$.

%\end{aprop}

%\section*{About the author:}
%   We would like a short biographical sketch,
%   beyond just your affiliation to be placed
%   after the bibliography.
%   And below that, your full address.

%\subsection*{Primus Scriber}
%   College of the Enlightenment,
%   Philadelphia, Pennsylvania, 42345-6543$\pm\epsilon$.
%   pscriber@cenet.edu

%\subsection*{Theco Author}~
%   Department of Statistics,
%   The Virtual University,
%   New York, NY 13291-5555.
%   also@aol.com

\end{document}